\documentclass[11pt,a4paper,reqno]{amsart} 
\usepackage[utf8]{inputenc} 
\usepackage[top=3.2cm, bottom=3.6cm, left=2.6cm, right=2.6cm]{geometry}
\usepackage{amsmath,amsthm,amsfonts,amstext,amssymb}
\usepackage{enumerate}
\usepackage{hyperref}
\usepackage{graphicx}
\usepackage{xcolor}
\hypersetup{
    colorlinks,
    linkcolor={blue!50!black},
    citecolor={blue!50!black},
    urlcolor={blue!80!black}
}

\usepackage{etoolbox}
\patchcmd{\section}{\scshape}{\bfseries\large}{}{}
\patchcmd{\subsubsection}{\itshape}{}{}{}
\makeatletter
\def\@seccntformat#1{\csname the#1\endcsname.\space}
\makeatother
\usepackage{tikz}
\usepackage[mode=tex]{standalone}
\usepackage{subcaption}

\newcommand{\treesubfigure}[3]{%
    \begin{subfigure}[t]{.23\textwidth}
        \vspace{0pt}
        \centering
        \parbox[c][2cm][c]{\linewidth}{%
            \centering
            \includegraphics[scale=.7]{#1}%
        }
        \caption{#2}
        \label{#3}
    \end{subfigure}%
}

\newtheorem{theorem}{Theorem}[section]
\newtheorem{corollary}[theorem]{Corollary}
\newtheorem{lemma}[theorem]{Lemma}
\newtheorem{proposition}[theorem]{Proposition}
\newtheorem{question}[theorem]{Question}

\theoremstyle{definition}

\numberwithin{equation}{section}

\usepackage{booktabs}
\usepackage{mathtools}
\newcommand{\E}{\mathbb{E}}
\newcommand{\N}{\mathbb{N}}
\renewcommand{\P}{\mathbb{P}}

\newcommand{\R}{\mathbb{R}}

\newcommand{\limn}{\lim\limits_{n \rightarrow \infty}}
\newcommand{\nconv}{\xrightarrow{n \to \infty}}
\newcommand{\Var}{\ensuremath{\mathrm{Var}}}
\newcommand{\Li}{\operatorname{Li}_2}
\newcommand{\IP}{\operatorname{(IP)}}

\newcommand{\ERG}{Erdős--R\'enyi}

\title{Fixed forests in the minimum spanning tree and cubic volume growth}
\author[L. Makowiec]{Luca Makowiec}
\address{University of Leipzig\\
Department of Mathematics\\
Augustusplatz 10, 04109 Leipzig, Germany.}
\email{luca.makowiec@uni-leipzig.de}
\date{\today}

\begin{document}


\begin{abstract}
Let $M_n$ be the minimum spanning tree of the complete graph $K_n$ with i.i.d.\ uniform edge weights. For a fixed forest $F$ with connected components $T_1, \ldots, T_d$, we show that there exists a function $\Psi$ on finite trees such that
\begin{equation*}
    n^{|E(F)|} \P_n(F \subseteq M_n) \longrightarrow \prod_{i=1}^d \Psi(T_i).
\end{equation*}
We give a recursive description of $\Psi$ and calculate it explicitly for several small trees. For the star $S_k$ and the path $P_k$, we prove that $\Psi(S_k) \sim \zeta(2)^k$ and $\Psi(P_k) \sim k^2/12$, respectively. We also show that the expected size of a ball of radius $r$ is asymptotic to $r^3/36$, and give exponential tail bounds.
\end{abstract}

\maketitle

%\tableofcontents

\section{Introduction}

Let $K_n$ be the complete graph on $[n]$, and assign i.i.d.\ continuous weights $w = (w_e)_{e \in E(K_n)}$ to its edges. Since the minimum spanning tree depends only on the ordering of the edge weights, its law is the same for every continuous weight distribution. We therefore assume throughout that the weights are uniform on $[0,1]$. The minimum spanning tree (MST) $M_n$ is the almost surely unique spanning tree satisfying
\begin{equation*}
    M_n = \arg\min_T \sum_{e \in E(T)} w_e,
\end{equation*}
where the minimum is taken over all spanning trees of $K_n$.

Exact probabilities for the entire tree $M_n$ are generally difficult to compute. Instead, we study the limiting inclusion probabilities of fixed finite forests. More precisely, for a fixed forest $F$, we consider
\begin{equation*}
    \limn n^{|E(F)|}\P_n(F \subseteq M_n).
\end{equation*}
Our main result shows that this limit factorizes over the connected components of $F$. The limiting contribution of each component is described by a function $\Psi$ on finite trees, for which we later give a recursive representation. As an application, we prove cubic volume growth: the limiting expected size of a radius-$r$ ball in $M_n$ is asymptotic to $r^3/36$ as $r \to \infty$.

%%%%%%%%%%%%%%%%%%%%%%%%%%
\subsection{Results} \label{SS:results}

For a fixed tree $T$, let $N_n(T)$ denote the number of injective embeddings of $T$ into $M_n$. We refer to Section~\ref{SS:notation} for a formal introduction to the notation used throughout. Our main result is the following.
\begin{theorem} \label{T:forest_formula}
    Let $F$ be a fixed forest with $m = |E(F)|$ and connected components $T_1, \ldots, T_d$. There exists a function $\Psi$ on the set of finite trees such that
    \begin{equation} \label{eq:forest_formula}
        \limn n^m \P_n(F \subseteq M_n) = \prod_{i=1}^d \Psi(T_i).
    \end{equation}
    Furthermore, for every fixed tree $T$ and every $0 < p < \infty$,
    \begin{equation} \label{eq:tree_count}
        \frac{N_n(T)}{n} \xrightarrow{L^p} \Psi(T).
    \end{equation}
\end{theorem}
\noindent We give a first description of $\Psi$ in Section~\ref{SS:Psi} and a recursive description in Section~\ref{S:compute_Psi}. In Figure~\ref{fig:trees_up_to_five_edges} we give values for $\Psi$ on trees up to $5$ edges. We list several corollaries that follow from Theorem~\ref{T:forest_formula}. 

Formula \eqref{eq:forest_formula} immediately implies that fixed trees on disjoint vertex sets are asymptotically independent.
\begin{corollary}
    If $F$ has connected components $T_1, \ldots, T_d$, then
    \begin{equation*}
        \P_n(F \subseteq M_n) = (1 + o_{|E(F)|}(1)) \prod_{i=1}^d \P_n(T_i \subseteq M_n).
    \end{equation*}
\end{corollary}

Let $S_k$ and $P_k$ denote the star and the path with $k$ edges, respectively. We have
\begin{equation} \label{eq:ball_as_prob}
    \E_n\bigl[|B_{M_n}(1,r)|\bigr] = \sum_{k=1}^r (n - 1)_k \P_n(P_k \subseteq M_n),
\end{equation}
so that the limit of expected ball sizes can be written as a sum of $\Psi$ on paths. %Here, for convenience's sake, we exclude $1$ from $B_{M_n}(1,r)$. 
It is known that the MST converges in a local weak sense to a limit $\mathcal{M}$ rooted at $\emptyset$, see \cite{Add13,NT24} and the overview in Section~\ref{SS:lit_overview}. For every fixed $r$, local convergence gives convergence in distribution of the corresponding ball sizes. Theorem~\ref{T:forest_formula} implies (see \eqref{eq:ball_factorial}) that the ball sizes have finite $k$-th moments for any $k \in \N$, which gives the following.
\begin{corollary} \label{C:converge_E}
    For every fixed $r \in \N$,
    \begin{equation*}
        \limn \E_n\bigl[|B_{M_n}(1,r)|\bigr] = \E\bigl[|B_{\mathcal{M}}(\emptyset,r)|\bigr] = \sum_{k=1}^r \Psi(P_k).
    \end{equation*}
\end{corollary}

Furthermore, by a calculation of $\Psi(P_2) = \Psi(S_2)$, the limiting variance of the degree of a fixed vertex can be computed explicitly.
\begin{corollary} \label{C:var_deg}
    For every fixed vertex $v$,
    \begin{equation*}
        \Var\bigl(\deg_{M_n}(v)\bigr) \xrightarrow{n \to \infty} 6 - 4\zeta(3) \approx 1.1917723\ldots.
    \end{equation*}
\end{corollary}
\noindent Interestingly, $\zeta(3)$ also appears as the limiting total weight of the MST~\cite{Fri85}. For comparison, the degree of a fixed vertex in the uniform spanning tree of $K_n$ has variance converging to $1$. The recursive description of $\Psi$ gives asymptotic formulas for stars $S_k$ and paths $P_k$. 
\begin{proposition} \label{P:star_path}
    As $k \to \infty$,
    \begin{align}
        \Psi(S_k) &= (1 + o(1))\zeta(2)^k, \label{eq:star_asymp}\\
        \Psi(P_k) &= (1 + o(1))\frac{k^2}{12}. \label{eq:path_asymp}
    \end{align}
\end{proposition}

\noindent In view of \eqref{eq:ball_as_prob} and Proposition~\ref{P:star_path}, one should expect cubic volume growth of the ball sizes. The following theorem confirms this prediction and establishes a concentration bound.
\begin{theorem} \label{T:ball_size}
    We have
    \begin{equation}
        \limn \E_n\bigl[|B_{M_n}(1,r)|\bigr] = \frac{r^3}{36}(1 + o(1)) \qquad \text{as } r \to \infty. \label{eq:ball_size}
    \end{equation}
    Furthermore, there exist constants $C,c > 0$ such that, for every fixed $r \geq 1$ and every $\lambda \geq 0$,
    \begin{align}
        \limsup_{n \to \infty}\P_n\bigl(|B_{M_n}(1,r)| \geq \lambda r^3\bigr) &\leq Ce^{-c\lambda}, \label{eq:ball_exp_bound}\\
        \limsup_{n \to \infty}\P_n\bigl(\bigl||B_{M_n}(1,r)| - \E_n|B_{M_n}(1,r)|\bigr| \geq \lambda r^3\bigr) &\leq Ce^{-c\lambda}. \label{eq:ball_concentration}
    \end{align}
\end{theorem}
\noindent
Together with local convergence and the Borel--Cantelli lemma, \eqref{eq:ball_exp_bound} implies that almost surely
\begin{equation*}
    \limsup_{r \to \infty}\frac{|B_{\mathcal{M}}(\emptyset,r)|}{r^3\log r} < \infty,
\end{equation*}
which improves the known upper logarithmic correction of the growth of $\mathcal{M}$. It remains open to determine the exact order of the maximal fluctuations of $|B_{\mathcal{M}}(\emptyset,r)|/r^3$ as $r \to \infty$.

Theorem~\ref{T:ball_size} assumes that first $n\to \infty$ and then $r \to \infty$. It is natural to ask the following.
\begin{question}
    Let $r(n) \to \infty$. For which growth rates $r(n)$ does
    \begin{equation*}
         \E_n\bigl[|B_{M_n}(1,r(n))|\bigr]= \frac{r(n)^3}{36}(1 + o(1))
    \end{equation*}
    hold? Does it remain valid for $r(n) = o(n^{1/3})$?
\end{question}

\begin{figure}[ht]
    \centering
    \treesubfigure
        {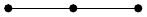}
        {$\Psi(S_2) = 8 - 4\zeta(3)$\\$\approx 3.1917\ldots$}
        {sfig:S2}
    \hfill
    \treesubfigure
        {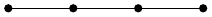}
        {$\Psi(P_3) = 52 - 36\zeta(3)$\\${}-4\zeta(5) \approx 4.5782\ldots$}
        {sfig:P3}
    \hfill
    \treesubfigure
        {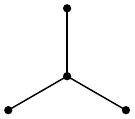}
        {$\Psi(S_3) = 108 - 48\zeta(3)$\\${}-36\zeta(4) - 6\zeta(5)$\\$\approx 5.1160\ldots$}
        {sfig:S3}
    \hfill
    \treesubfigure
        {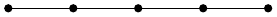}
        {$\Psi(P_4) \approx 6.1571\ldots$}
        {sfig:P4}

    \par\medskip

    \treesubfigure
        {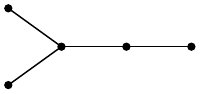}
        {$\Psi(B_4) \approx 7.1268\ldots$}
        {sfig:B4}
    \hfill
    \treesubfigure
        {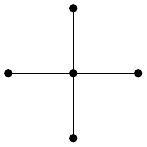}
        {$\Psi(S_4) \approx 8.2526\ldots$}
        {sfig:S4}
    \hfill
    \treesubfigure
        {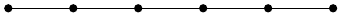}
        {$\Psi(P_5) \approx 7.925\ldots$}
        {sfig:P5}
    \hfill
    \treesubfigure
        {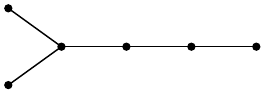}
        {$\Psi(T_{3,1,1}) \approx 9.3142\ldots$}
        {sfig:T311}

    \par\medskip

    \treesubfigure
        {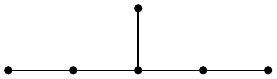}
        {$\Psi(T_{2,2,1}) \approx 9.6665\ldots$}
        {sfig:T221}
    \hfill
    \treesubfigure
        {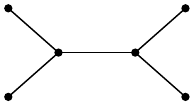}
        {$\Psi(D_{2,2}) \approx 10.9506\ldots$}
        {sfig:D22}
    \hfill
    \treesubfigure
        {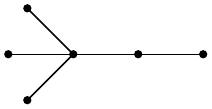}
        {$\Psi(B_5) \approx 11.3676\ldots$}
        {sfig:B5}
    \hfill
    \treesubfigure
        {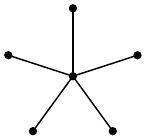}
        {$\Psi(S_5) \approx 13.3736\ldots$}
        {sfig:S5}

    \caption{The function $\Psi$ evaluated on the non-isomorphic trees with two to five edges. The values from \subref{sfig:P4} to \subref{sfig:S5} are obtained numerically. The code may be found at \url{https://github.com/L-makowiec/mst-tree-psi-computations}.}
    \label{fig:trees_up_to_five_edges}
\end{figure}

\subsection{Literature overview} \label{SS:lit_overview}

The MST is a fundamental problem in combinatorial optimization and network design. Classical greedy constructions are due to Kruskal~\cite{Kru56} and Prim~\cite{Pri57}. A foundational probabilistic result is Frieze's theorem~\cite{Fri85}: for i.i.d.\ uniform $[0,1]$ edge weights on $K_n$, the total weight of the MST converges in probability to $\zeta(3) \approx 1.2020569\ldots$. More generally, Frieze obtained the limit $\zeta(3)/F'(0)$ under moment assumptions on the edge-weight distribution function $F$, and Steele~\cite{Ste87} subsequently removed these assumptions.

\smallskip 
The results most closely related to the present work concern the local structure of the MST. Aldous~\cite{Ald90} proved convergence of the finite fringe subtree at a fixed vertex and obtained an explicit expression for the limiting degree distribution. Aldous and Steele~\cite[Theorem~5.4]{AS04} subsequently identified the weighted local limit for regular graph sequences whose degrees tend to infinity. For $K_n$, Addario-Berry~\cite{Add13} described this limit explicitly as the component $\mathcal{M}$ containing the root in the wired minimal spanning forest on the Poisson-weighted infinite tree, constructed from invasion percolation clusters. He also proved that $\mathcal{M}$ has cubic volume growth up to explicit subpolynomial corrections. Nachmias and Tang~\cite[Theorem~6.1]{NT24} established convergence in the unweighted Benjamini--Schramm topology: for every fixed $r$, and rooted tree $t$ of height $r$
\begin{equation} \label{eq:local_conv}
    \P_n( B_{M_n}(1,r) \simeq t) \nconv \P( B_{\mathcal{M}}(\emptyset,r) \simeq t).
\end{equation}
We remark that this convergence is not strong enough to yield the result in Theorem~\ref{T:forest_formula} even for a single tree. They also improved the volume-growth corrections to powers of $\log r$ and showed that $\mathcal{M}$ has spectral dimension $3/2$.

Tang and Zhang~\cite{TZ26} proved that any two non-adjacent edges of the MST of $K_n$ are negatively correlated for all sufficiently large $n$. They asked in \cite[Question~3.8]{TZ26} whether the second moment of the degree of a fixed vertex converges to that of the root degree in $\mathcal{M}$. Corollary~\ref{C:var_deg} answers this question affirmatively. By \cite[Corollary~2.4]{TZ26}, this also yields negative correlation for adjacent edges when $n$ is sufficiently large. In fact, the recent work \cite{Gup26} provides, among other results, a proof of this negative correlation and an explanation for the appearance of the limiting weight $\zeta(3)$ in the variance of the degree. 

\smallskip 

At the global scale, Addario-Berry, Broutin and Reed~\cite{ABR09} proved that the diameter of $M_n$ is of order $n^{1/3}$. Addario-Berry, Broutin, Goldschmidt and Miermont~\cite{ABGM17} subsequently proved convergence, after rescaling graph distances by $n^{-1/3}$, to a random compact measured binary $\mathbb R$-tree of Minkowski dimension three. Free and wired minimal spanning forests on infinite graphs are studied systematically in \cite{LPS06}.

\smallskip 
The uniform spanning tree (UST) provides a useful comparison with the MST. See \cite[Chapter~4]{LP16} for a general introduction and \cite{NP22} for local limits on regular graphs of diverging degree. For every fixed labeled tree $T$ with $m$ edges, $\P(T \subseteq \mathrm{UST}(K_n)) = (m + 1)/n^m$, which yields quadratic expected volume growth. The local limit of $\mathrm{UST}(K_n)$ is the Poisson$(1)$ Galton--Watson tree conditioned to survive.

Another recently studied model is the random spanning tree in random environment, which interpolates between the UST and the MST through an inverse-temperature parameter $\beta$; see \cite{MSS26,MSS24,Kus24,Mak25}. It was shown in \cite{Mak24,Kus25} that its local limit agrees with that of the UST when $\beta = o(n)$ and with that of the MST when $\beta$ is much bigger than $n\log n$. The behavior when $\beta \asymp n$, including the corresponding volume growth, remains open. It is natural to conjecture that all the results in Section~\ref{SS:results} also hold for this model whenever $\beta$ is much larger than $n \log n$.

%%%%

\subsection{Proof ideas} 

The proof of Theorem~\ref{T:forest_formula} begins with joint convergence of partitions induced by the \ERG{} random graph process with $p = t/n$. For $t \leq 1$, the distinguished vertices lie in separate components with probability tending to one, while for $t > 1$, their connectivity is asymptotically determined by their membership in the giant component. Given a fixed tree $T$, we assign to each $v \in V(T)$ an independent random variable $A_v$ satisfying $\P(A_v \leq t) = \theta(t)$, where $\theta(t)$ is the limiting proportion of vertices in the giant component of $G_n(t/n)$. At time $t$, the vertices satisfying $A_v \leq t$ form one block, while the remaining vertices are singletons. Exchangeability and the asymptotic size and nesting of the giant components identify these as the limiting partitions. 

Heuristically, the tree $T$ is contained in the MST if for every edge $e \in E(T)$, at time $p = w_e \approx t/n$, at least one of the two components of $T \setminus \{e\}$ contains no vertex that has yet joined the giant component. Otherwise, Kruskal's algorithm will eventually reject an edge for which the two endpoints are already connected to the giant. This is reminiscent of Aldous's construction in \cite{Ald90}, in which the component of a fixed vertex is grown until it first joins the giant component, and the branch containing this connection is then removed. The cycle property of the MST then expresses the inclusion of a fixed tree in terms of partitions induced by the \ERG{} process, yielding the limit $\Psi(T)$. Exchangeability, concentration, and uniform moment bounds give convergence of the normalized tree counts, while the contribution from overlapping embeddings is negligible, extending the result from trees to forests.

To analyze $\Psi(T)$, we condition on the vertex attaining the minimum of $(A_v)_{v \in V(T)}$. Independence between the resulting branches gives a recursion for $\Psi$ through functions $F_{T_o}$ associated with rooted graphs $T_o$, which themselves also satisfy a corresponding recursion. This yields the formulas for small trees and the asymptotics for stars. For paths, the functions $F_{P_n}$ are interpreted as the densities of a decreasing Markov chain, whose scaling limit gives $\Psi(P_k) \sim k^2/12$. Summing the path asymptotics yields cubic expected volume growth, while the exponential upper tail follows by expressing the factorial moments of the ball size as sums of $\Psi(T)$ over a finite set of trees and applying corresponding combinatorial bounds.

\subsection{Outline} 

Section~\ref{S:prelim} introduces the notation and the required background on minimum spanning trees, \ERG{} graphs, the function $\Psi$, and partitions. We then prove \eqref{eq:forest_formula} for a single tree in Section~\ref{S:proof_forest_formula}, establish \eqref{eq:tree_count}, and use it to extend the result to forests. Section~\ref{S:compute_Psi} gives a recursive description of $\Psi$, from which we derive the formulas for small trees and the asymptotics for stars and paths. Finally, Section~\ref{S:concentration} contains the proof of Theorem~\ref{T:ball_size}, while Appendix~\ref{S:Appendix} briefly illustrates how some explicit values can be obtained.

%%%%%%%%%%%%%%%%%%%%%%%%%%%%%%%%%%%%%%%%%%%%%%%%

\section{Preliminaries} \label{S:prelim}

\subsection{Notation} \label{SS:notation}

For each $n \geq 1$, let $[n] := \{1,\ldots,n\}$. We write $\P_n$ for the law of the i.i.d.\ uniform on $[0,1]$ edge weights $(w_e)_{e \in E(K_n)}$ and $\E_n$ for the corresponding expectation. Let $M_n$ denote the MST determined by these almost surely distinct weights, as defined in Section~\ref{SS:MST}. We later use $\P_A$ and $\E_A$ for the law and expectation of the limiting variables introduced in \eqref{eq:Law_A}.

For a finite graph $G$, we write $V(G)$ and $E(G)$ for its vertex and edge sets, $d_G(u,v)$ for the graph distance between $u$ and $v$, and $\deg_G(v)$ for the degree of $v$. We write $u \xleftrightarrow{G} v$ if $u$ and $v$ belong to the same connected component of $G$. For $A \subseteq E(G)$ and $W \subseteq V(G)$, we set
\begin{equation*}
G \setminus A := \bigl(V(G),E(G) \setminus A\bigr)
\end{equation*}
and $G \setminus W$ as the graph induced by the vertices in $V(G) \setminus W$. For $v \in V(G)$ and $r \geq 1$, define the ball of radius $r$ centered at $v$, excluding its center, by
\begin{equation*}
B_G(v,r) := \{u \in V(G) \setminus \{v\} : d_G(u,v) \leq r\}.
\end{equation*}
The exclusion of the center is merely a notational convenience\footnote{In \eqref{eq:local_conv}, the ball is viewed as a subgraph including the center.}, as it avoids additional factors when computing moments of the size of the ball.

All fixed trees and forests are finite abstract graphs independent of $n$. When such a graph is assigned distinct labels in $[n]$, the event $T \subseteq M_n$ means that this fixed labeled copy of $T$ is contained in the MST $M_n$. By vertex exchangeability, this probability is independent of the chosen labeling, and, by a slight abuse of notation, we denote its common value by $\P_n(T \subseteq M_n)$. For a tree $T$ and $u,v \in V(T)$, let $P_T(u,v)$ denote the unique path from $u$ to $v$ in $T$. For an edge $e \in E(T)$ with endpoints denoted by $e^-$ and $e^+$, let $T_e^-$ and $T_e^+$ be the components of $T \setminus \{e\}$ containing $e^-$ and $e^+$, respectively. For $k \geq 1$, let $S_k$ and $P_k$ denote the star and the path with $k$ edges.

For a fixed tree $T$, let
\begin{equation*}
N_n(T) := \bigl|\{f \colon V(T) \hookrightarrow [n] : \{f(u),f(v)\} \in E(M_n) \text{ for every } \{u,v\} \in E(T)\}\bigr|
\end{equation*}
be the number of injective embeddings of $T$ into $M_n$. For $m \geq 1$ and $x \in \mathbb R$, the falling factorial is defined as
\begin{equation*}
(x)_m := x(x - 1)\cdots(x - m + 1),
\end{equation*}
with $(x)_0 := 1$ and $(x)_m = 0$ whenever $x$ is a non-negative integer with $x < m$. Thus, if $T$ has $q$ vertices, vertex exchangeability gives
\begin{equation*}
\E_n[N_n(T)] = (n)_q \P_n(T \subseteq M_n).
\end{equation*}

Finally, for nonnegative functions $f$ and $g$, we write $f(n) = O(g(n))$ if $|f(n)| \leq Cg(n)$ for all sufficiently large $n$ and some $C > 0$, and $f(n) = o(g(n))$ if $f(n)/g(n) \to 0$. A subscript indicates permitted dependence of the implicit constant; for example, $f(n) = O_k(g(n))$ means that the constant may depend on $k$.

%%%%%%%%%%%%%%%%%%%%%%%%%%%%%%%%%%%%%%%%%
\smallskip
\subsection{Minimum spanning tree} \label{SS:MST}

Given a connected graph $G = (V,E)$, let $w = (w_e)_{e \in E(G)}$ be distinct edge weights. Recall that the MST is the tree that minimizes $\sum_{e \in E(T)} w_e$ among all spanning trees (connected cycle-free subgraphs of $G$). The two most well-known algorithms to construct the MST are Prim's algorithm~\cite{Pri57} and Kruskal's algorithm~\cite{Kru56}. We shall use the latter and briefly state it here. 

\smallskip
\textbf{Kruskal's algorithm:} Let $|E(G)| = m$, and sort the edges $e_1, \ldots, e_m$ by weight: $w_{e_1} < w_{e_2} < \ldots < w_{e_m}$, and examine the edges in this order. Let $G_i$ be the graph after the algorithm has examined $e_1, \ldots, e_i$ (with $V(G_0) = V(G)$ and $E(G_0) = \emptyset$). If the endpoints of $e_{i+1}$ are not connected in $G_i$, add $e_{i+1}$ to $G_i$ to obtain $G_{i+1}$. Otherwise, let $G_{i+1} = G_i$. The graph $G_m$ is the MST. 
\smallskip 

Kruskal's algorithm gives the following characterization.
\begin{lemma} \label{L:MST_cycle}
    An edge $e=\{u,v\}$ is rejected by Kruskal's algorithm if and only if there exists a path from $u$ to $v$ consisting entirely of edges with weight strictly smaller than $w_e$. Equivalently, $e$ is rejected if and only if it is the unique edge of maximal weight in some cycle of $G$.
\end{lemma}

%%%%%%%%%%%%%%%%%%%%%%%%%%%%%%%%
\smallskip
\subsection{\ERG{} graphs} \label{SS:ERG}

For $p \in [0,1]$, let $G_n(p)$ be the \ERG{} random graph obtained by retaining each edge of $K_n$ independently with probability $p$. A retained edge is called open, and every other edge is called closed. We regard $p$ as time in the random graph process $(G_n(p))_{p \in [0,1]}$. Let $\mathbf C_n(p)$ be the collection of connected components of $G_n(p)$, and write $\mathcal{C}_1^n(p),\mathcal{C}_2^n(p),\ldots$ for these components in decreasing order of size. We identify the components with their vertex sets. To break ties in a permutation-invariant way\footnote{Usually one chooses the component with the smallest vertex label, which is not permutation-invariant.}, attach an independent uniform random variable to each vertex in $K_n$ and, among components of the same size, place first the component containing the vertex with the smallest attached value.

We will mainly consider $p = t/n$ for fixed $t > 0$. It is a classical result of Erdős and Rényi~\cite{ER60} that the \ERG{} random graph undergoes a phase transition at $t = 1$: for $t \leq 1$, every component has sublinear size, whereas for $t > 1$, the largest component has linear size and every other component has sublinear size. For $t > 1$, let $\theta(t)$ be the unique positive solution of
\begin{equation*}
    1 - \theta(t) = e^{-t\theta(t)}.
\end{equation*}
Equivalently, $\theta(t)$ is the survival probability of a Galton--Watson tree with Poisson($t$) offspring. For convenience's sake, we set $\theta(t) = 0$ for $t \leq 1$. We will make use of the following well-known theorem, see e.g.\ Chapter~4 of \cite{vdH17}.

\begin{theorem} \label{T:ERG_sizes}
    For every fixed $t > 0$,
    \begin{equation*}
        \frac{|\mathcal{C}_1^n(t/n)|}{n} \xrightarrow{\P} \theta(t) \qquad \text{and} \qquad \frac{|\mathcal{C}_2^n(t/n)|}{n} \xrightarrow{\P} 0.
    \end{equation*}
    % For every fixed $t \leq 1$,
    % \begin{equation*}
    %     \frac{|\mathcal C_1^n(t/n)|}{n} \xrightarrow{\P} 0.
    % \end{equation*}
\end{theorem}

Recall that we choose edge weights $w_e \sim U[0,1]$ on the complete graph $K_n$. We couple these weights and the \ERG{} process by declaring
\begin{equation*}
    e \in E(G_n(p)) \quad\Longleftrightarrow\quad w_e \leq p.
\end{equation*}

\smallskip
\subsection{First description of \texorpdfstring{$\Psi$}{Psi}} \label{SS:Psi} 

The function $\Psi$ is defined in terms of the asymptotic density $\theta(t)$ of the giant component. Under the scaling $p = t/n$, the time at which a fixed vertex joins the giant component has asymptotic distribution function $\theta(t)$. Accordingly, let $T$ be a tree with vertex set $[q]$, and let $(A_i)_{i \in [q]}$ be i.i.d.\ random variables satisfying
\begin{equation} \label{eq:Law_A}
\P_A(A_1 \leq t) = \theta(t),
\end{equation}
with $A_1 > 1$ almost surely. Assign to each edge $e \in E(T)$ a nonnegative value $x_e$, corresponding to the scaled weight $n w_e$ of the associated edge in the original complete graph. Conditional on a realization of $(A_i)_{i \in [q]}$, we can define a process that accepts $T$ if
\begin{equation} \label{eq:accept_T}
\max_{e \in P_T(i,j)} x_e < \max\{A_i,A_j\} \quad \text{for all } i \neq j.
\end{equation}
In view of the cycle property in Lemma~\ref{L:MST_cycle}, this criterion may be interpreted as a limiting version of Kruskal's algorithm: $\max\{A_i,A_j\}$ represents the first time at which both $i$ and $j$ are connected to the giant component (outside $T$), and for every edge of $T$ to be accepted, all edges on the path $P_T(i,j)$ must have smaller weight than this connection time.

Recall that, for $e \in E(T)$, we let $T_e^-$ and $T_e^+$ be the two connected components of $T \setminus \{e\}$. In the preceding description, the first time at which each of these components contains a vertex in the giant component is
\begin{equation} \label{eq:def_tau}
\tau_e(T) := \max\Big\{\min_{u \in V(T_e^-)} A_u,\min_{v \in V(T_e^+)} A_v\Big\} = \min_{u \in V(T_e^-),\,v \in V(T_e^+)}\max\{A_u,A_v\}.
\end{equation}
We have the following correspondence between $\tau_e(T)$ and the acceptance criterion in \eqref{eq:accept_T}.

\begin{lemma}  \label{L:equiv_accept}
For every $(x_e)_{e \in E(T)}$, it holds that
\begin{equation*}
\max_{e \in P_T(i,j)} x_e < \max\{A_i,A_j\} \ \forall i,j, \ i\neq j \quad \Longleftrightarrow \quad x_e < \tau_e(T) \ \forall e \in E(T).
\end{equation*}
\end{lemma}

\begin{proof}
 Assume the left-hand side holds, and fix $e \in E(T)$. For any $i \in V(T_e^-),j \in V(T_e^+)$, we have $x_e \leq \max_{e \in P_T(i,j)} x_e < \max(A_i, A_j)$. Taking the minimum over $i$ and $j$ gives the right-hand side. If the right-hand side holds, fix $i,j$ and let $e \in P_T(i,j)$. Since $i$ and $j$ must lie in different components of $T \setminus \{e\}$, we have $x_e < \tau_e(T) \leq \max\{ A_i, A_j\}$. Taking the maximum over all edges $e \in P_T(i,j)$ gives the left-hand side.
\end{proof}

\noindent Conditional on $(A_i)_{i \in [q]}$, the tree $T$ is therefore accepted precisely when
\begin{equation*}
(x_e)_{e \in E(T)} \in \prod_{e \in E(T)}[0,\tau_e(T)).
\end{equation*}
Averaging over $(A_i)_{i \in [q]}$ motivates the definition
\begin{equation} \label{eq:def_Psi}
\Psi(T) := \E_A\left[\prod_{e \in E(T)}\tau_e(T)\right].
\end{equation}

%%%%%%%%%%%%%%%%%%%%%%%%%%%%
\smallskip
\subsection{Partitions} \label{SS:partitions}
A partition $\pi$ of a finite set $S$ is a collection of nonempty, pairwise disjoint subsets of $S$ whose union is $S$. We call these subsets the blocks of $\pi$. For $i,j \in S$, we write $i \sim_\pi j$ if $i$ and $j$ belong to the same block of $\pi$.

Given a realization of random variables $(A_i)_{i \in [q]}$ as in \eqref{eq:Law_A} and finite times $0 < t_1 < \ldots < t_\ell$, define a collection of partitions
\begin{equation*}
    \bigl(\Pi^q(t_1),\ldots,\Pi^q(t_\ell)\bigr),
\end{equation*}
where each $\Pi^q(t_k)$ is a partition of $[q]$, by saying that
\begin{equation*}
    i \sim_{\Pi^q(t_k)} j
    \quad\Longleftrightarrow\quad
    i=j
    \quad\text{or}\quad
    A_i,A_j\leq t_k.
\end{equation*}
Thus, all $i \in [q]$ satisfying $A_i \leq t_k$ form one block, while every other element is a singleton. In particular, $\Pi^q(t_k)$ almost surely consists only of singletons whenever $t_k \leq 1$. 

For the complete graph, we shall define partitions in a similar manner but allow for the removal of some edges. Let $n\geq q$ and consider the vertices $1,\ldots,q$ of $K_n$, and let $F \subseteq E(K_q)$ be a finite fixed edge set of $K_n$ (we view $K_q$ as a subgraph of $K_n$ for $q \leq n$). Given finite times $0<t_1<\ldots<t_\ell$, the \ERG{} random graph process defines a sequence of partitions
\begin{equation*}
    \bigl(\Pi^{q,F}_n(t_1),\ldots,\Pi^{q,F}_n(t_\ell)\bigr),
\end{equation*}
where each $\Pi^{q,F}_n(t_k)$ is a partition of $[q]$, by saying that
\begin{equation*}
    i \sim_{\Pi^{q,F}_n(t_k)} j
    \quad\Longleftrightarrow\quad
    i \xleftrightarrow{G_n(t_k/n)\setminus F} j.
\end{equation*}
That is, $i$ and $j$ are in the same block of $\Pi^{q,F}_n(t_k)$ if and only if they are connected in the Erdős--Rényi random graph at time $t_k/n$ without using any edges from $F$. In Lemma~\ref{L:partition_convergence}, we show that these partitions converge in distribution to those partitions defined by $(A_i)_{i \in [q]}$.

%%%%%%%

\section{Proof of Theorem~\ref{T:forest_formula}} \label{S:proof_forest_formula} 

We prove Theorem~\ref{T:forest_formula} in four steps, each building on the previous one. First, we show convergence of the partitions induced by the \ERG{} process; then we show equation~\eqref{eq:forest_formula} for a single tree; then we prove \eqref{eq:tree_count} for a tree; and finally we extend Theorem~\ref{T:forest_formula} to a forest.
 
\subsection{Convergence of partitions}

The proof begins with the joint convergence of the partitions introduced in Section~\ref{SS:partitions}.

\begin{lemma} \label{L:partition_convergence}
    Fix $q,\ell \geq 1$ and an edge set $F \subseteq E(K_q)$. Then, for any fixed times $0 < t_1 < \cdots < t_\ell$,
    \begin{equation*}
        \bigl(\Pi^{q,F}_n(t_1), \ldots, \Pi^{q,F}_n(t_\ell)\bigr) \xrightarrow{(d)} \bigl(\Pi^{q}(t_1), \ldots, \Pi^{q}(t_\ell)\bigr).
    \end{equation*}
\end{lemma}

\begin{proof}
    We first reduce to the case $F=\emptyset$. The two sequences of partitions
    \begin{equation*}
        \bigl(\Pi^{q,F}_n(t_1),\ldots,\Pi^{q,F}_n(t_\ell)\bigr)
        \qquad\text{and}\qquad
        \bigl(\Pi^{q,\emptyset}_n(t_1),\ldots,\Pi^{q,\emptyset}_n(t_\ell)\bigr)
    \end{equation*}
    agree whenever no edge of $F$ is present in $G_n(t_\ell/n)$. Hence,
    \begin{equation*}
        \P_n\left(
        \bigl(\Pi^{q,F}_n(t_1),\ldots,\Pi^{q,F}_n(t_\ell)\bigr)
        \neq
        \bigl(\Pi^{q,\emptyset}_n(t_1),\ldots,\Pi^{q,\emptyset}_n(t_\ell)\bigr)
        \right)
        \leq \frac{|F|t_\ell}{n}=o(1).
    \end{equation*}
    It is therefore enough to prove the result when $F=\emptyset$, and we write $\Pi^q_n=\Pi^{q,\emptyset}_n$.

    We first consider times $t_k\leq1$. For $i\neq j$, exchangeability gives
    \begin{align}
        \P_n\left(i\xleftrightarrow{G_n(t_k/n)}j\right)
        &= \frac{1}{n(n-1)}\E_n\big[ \sum_{ u} \sum_{v \neq u} \mathbf{1}_{u,v \text{ are in the same component}} \big] \nonumber \\
        &=
        \frac{1}{n(n-1)} \E_n\Big[ \sum_{\mathcal{C}\in\mathbf{C}_n(t_k/n)} |\mathcal{C}|(|\mathcal{C}|-1)
        \Big]  \nonumber \\
        &\leq \E_n\big[ \frac{|\mathcal{C}_1^n(t_k/n)|}{n}
        \big] \label{eq:bound_sub_crit_con}.
    \end{align}
    By Theorem~\ref{T:ERG_sizes}, $|\mathcal{C}_1^n(t_k/n)|/n \xrightarrow{\P} 0$. Since this sequence is bounded, the last expectation converges to zero. A union bound over the finitely many pairs $i,j \in [q]$ and the finitely many times $t_k \leq 1$ shows that, jointly at all such times, $\Pi^q_n(t_k)$ consists only of singletons with probability tending to one. The same is almost surely true for $\Pi^q(t_k)$ because $A_i > 1$ almost surely. Without loss of generality, we may therefore assume for the remainder of the proof that $1 < t_1 < \cdots < t_\ell$.

    Define the ``good'' nesting event 
    \begin{equation*}
        \mathcal{N}_n := \big\{ \mathcal{C}_1^n(t_1/n)\subseteq\cdots\subseteq\mathcal{C}_1^n(t_\ell/n)\big\} 
    \end{equation*}
    and the ``bad'' event
    \begin{equation*}
        \mathcal{B}_n:=
        \big\{
            \exists\, k\in[\ell],\ 
            \exists\, i,j\in[q],\ i\neq j:
            i\xleftrightarrow{
                G_n(t_k/n)\setminus\mathcal{C}_1^n(t_k/n)
                }j
        \big\}
    \end{equation*}
    that some of the distinguished vertices are connected outside the largest component. As any component in $G_n(t_k/n)$ is contained in a component of $G_n(t_{k + 1}/n)$, if $\mathcal{C}_1^n(t_k/n)$ were not contained in $\mathcal{C}_1^n(t_{k + 1}/n)$, then
    \begin{equation*}
        |\mathcal{C}_1^n(t_k/n)| \leq |\mathcal{C}_2^n(t_{k + 1}/n)|.
    \end{equation*}
    By Theorem~\ref{T:ERG_sizes}, the left-hand side divided by $n$ converges in probability to $\theta(t_k) > 0$, whereas the right-hand side divided by $n$ converges in probability to $0$. A union bound over the finitely many values of $k$ therefore gives
    \begin{equation*}
        \P_n(\mathcal{N}_n) = 1 - o(1).
    \end{equation*}
    Similarly as in \eqref{eq:bound_sub_crit_con}, for every $i\neq j$,
    \begin{equation*}
        \P_n\big( i \xleftrightarrow{G_n(t_k/n)\setminus\mathcal{C}_1^n(t_k/n)} j \big) \leq \E_n\Big[\frac{|\mathcal{C}_2^n(t_k/n)|}{n}\Big]
        =o(1),
    \end{equation*}
    where the last equality follows from Theorem~\ref{T:ERG_sizes}. Hence, by a union bound,
    \begin{equation*}
        \P_n(\mathcal{B}_n)=o(1).
    \end{equation*}

    On $\mathcal{N}_n$, define the disjoint layers
    \begin{align*}
        D_1&:=\mathcal{C}_1^n(t_1/n),\\
        D_k&:=\mathcal{C}_1^n(t_k/n)
        \setminus\mathcal{C}_1^n(t_{k-1}/n),
        \qquad 2\leq k\leq\ell,\\
        D_\infty&:=[n]\setminus\mathcal{C}_1^n(t_\ell/n).
    \end{align*}
    By Theorem~\ref{T:ERG_sizes}, jointly,
    \begin{align*}
        \frac{|D_1|}{n}&\xrightarrow{\P}\theta(t_1)=:p_1,\\
        \frac{|D_k|}{n}&\xrightarrow{\P}
        \theta(t_k)-\theta(t_{k-1})=:p_k,
        \qquad 2\leq k\leq\ell,\\
        \frac{|D_\infty|}{n}&\xrightarrow{\P}
        1-\theta(t_\ell)=:p_\infty.
    \end{align*}
    Define the random layer assignment of the vertices $[q]$ by
    \begin{equation*}
        \mathbf L_n=(L_{n,i})_{i\in[q]},
        \qquad
        L_{n,i}=\tau \in \{1,\ldots,\ell,\infty\}
        \quad\Longleftrightarrow\quad
        i\in D_\tau,
    \end{equation*}
    and define $\mathbf L_n$ arbitrarily on $\mathcal{N}_n^c$. Similarly, define $\mathbf L=(L_i)_{i\in[q]}$ by
   \begin{equation*}
        L_i :=
        \begin{cases}
            1, & A_i\leq t_1,\\
            k, & t_{k-1}<A_i\leq t_k,
            \qquad 2\leq k\leq\ell,\\
            \infty, & A_i>t_\ell.
        \end{cases}
    \end{equation*}
    For $f = (f_i)_{i \in [q]}\in\{1,\ldots,\ell,\infty\}^q$, define partitions
    $\pi_k(f)$ of $[q]$ by
    \begin{equation*}
        i\sim_{\pi_k(f)}j
        \quad\Longleftrightarrow\quad
        i=j
        \quad\text{or}\quad
        f_i,f_j\leq k.
    \end{equation*}
    On $\{ \mathcal{N}_n\cap\mathcal{B}_n^c \}$, two distinct vertices
    $i,j\in[q]$ are connected in $G_n(t_k/n)$ if and only if they both
    belong to $\mathcal{C}_1^n(t_k/n)$. Consequently,
    \begin{equation*}
        \bigl(\Pi_n^q(t_1),\ldots,\Pi_n^q(t_\ell)\bigr)
        =
        \bigl(\pi_1(\mathbf L_n),\ldots,\pi_\ell(\mathbf L_n)\bigr)
    \end{equation*}
    on $ \{ \mathcal{N}_n\cap\mathcal{B}_n^c \}$. Moreover, from the
    definition of $\Pi^q$,
    \begin{equation*}
        \bigl(\Pi^q(t_1),\ldots,\Pi^q(t_\ell)\bigr)
        =
        \bigl(\pi_1(\mathbf L),\ldots,\pi_\ell(\mathbf L)\bigr),
    \end{equation*}
    almost surely. Since $\P_n(\mathcal{N}_n\cap\mathcal{B}_n^c) = 1 -o(1)$, it therefore suffices to prove that
    \begin{equation*}
        \mathbf L_n\xrightarrow{(d)}\mathbf L.
    \end{equation*}

    Since
    \begin{equation*}
        \P_A(L_i=\tau)=p_\tau,
        \qquad \tau\in\{1,\ldots,\ell,\infty\},
    \end{equation*}
    the independence of $A_1,\ldots,A_q$ gives
    \begin{equation*}
        \P_A(\mathbf L=f)=\prod_{i=1}^q p_{f_i}.
    \end{equation*}
    For $\tau\in\{1,\ldots,\ell,\infty\}$, let $m_\tau(f):= |\{i\in[q]:f_i=\tau\}|$. By exchangeability,
    \begin{align*}
        \P_n\bigl(\mathcal{N}_n\cap\{\mathbf L_n=f\}\bigr)
        &=
        \E_n\Big[ \frac{\mathbf{1}_{\mathcal{N}_n}}{(n)_q} \sum_{\substack{v_1,\ldots,v_q\in[n]\\\mathrm{distinct}}} \prod_{i=1}^q\mathbf{1}_{\{v_i\in D_{f_i}\}}
        \Big]\\
        &= \E_n\Big[\mathbf{1}_{\mathcal{N}_n}\frac{\prod_{\tau\in\{1,\ldots,\ell,\infty\}} (|D_\tau|)_{m_\tau(f)}} {(n)_q}
        \Big].
    \end{align*}
    Indeed, for each $\tau$, the coordinates $i$ satisfying $f_i=\tau$ must be filled by an ordered $m_\tau(f)$-tuple of distinct vertices of $D_\tau$, giving $(|D_\tau|)_{m_\tau(f)}$ choices, and, since the layers are disjoint, multiplying these counts gives exactly the number of ordered distinct $q$-tuples appearing in the sum. The joint convergence of the layer
    sizes gives
    \begin{equation*}
        \mathbf{1}_{\mathcal{N}_n}
        \frac{ \prod_{\tau\in\{1,\ldots,\ell,\infty\}}
        (|D_\tau|)_{m_\tau(f)}}
        {(n)_q}
        \xrightarrow{\P}
        \prod_{\tau\in\{1,\ldots,\ell,\infty\}}
        p_\tau^{m_\tau(f)}.
    \end{equation*}
    Since the random variable on the left is bounded, we obtain
    \begin{equation*}
        \P_n\bigl(\mathcal{N}_n\cap\{\mathbf L_n=f\}\bigr)
        \nconv
        \prod_{\tau\in\{1,\ldots,\ell,\infty\}}
        p_\tau^{m_\tau(f)}
        =
        \prod_{i=1}^q p_{f_i}
        =
        \P_A(\mathbf L=f).
    \end{equation*}
    As $\P_n(\mathcal{N}_n^c)=o(1)$, and the space $\{1,\ldots,\ell,\infty\}^q$ is finite, this completes the proof. 
\end{proof}

\subsection{Proof of (\ref{eq:forest_formula}) for a single tree}
We first show \eqref{eq:forest_formula} when $F = T$ is a single tree. Write $m = |E(T)|$ and label the $q = m + 1$ vertices of $T$ by $[q]$.
For $u,v \in V(T)$, $u \neq v$, define
\begin{equation*}
     R_n^T(u,v)=n \min_{\gamma:u\leftrightarrow v\text{ in }K_n\setminus E(T)} \max_{e\in E(\gamma)} w_e.
\end{equation*}
Thus $R_n^T(u,v)$ is the first rescaled time at which $u$ and $v$ are connected in $G_n(t/n) \setminus E(T)$. In particular, for every $t > 0$,
\begin{equation}\label{eq:R_partition_equivalence}
    t < R_n^T(u,v) \quad\Longleftrightarrow\quad u\text{ and }v\text{ belong to different blocks of }\Pi_n^{q,E(T)}(t).
\end{equation}

We first prove pointwise convergence of the conditional inclusion probabilities. Lemma~\ref{L:connection-tail} below will give an estimate allowing the dominated convergence theorem to be applied. 

 \begin{lemma} \label{L:pointwise_conv}
    Fix $x \in (0,\infty)^m$ with distinct entries. Then
    \begin{equation*}
         \limn \P_n ( T \subseteq M_n \mid w_e = x_e/n \ \forall e \in E(T)) = \P_A(x_e < \tau_e(T) \ \forall e \in E(T))
    \end{equation*}
\end{lemma}

\begin{proof}
Condition throughout on $w_e = x_e/n$ for every $e \in E(T)$. We claim that, almost surely,
\begin{equation} \label{eq:equiv_R^T}
    T \subseteq M_n \iff \max\limits_{e \in P_T(u,v)} x_e <R_n^T(u,v)  \ \forall u\neq v.
\end{equation}
First, suppose that there exist $u,v \in V(T)$, $u \neq v$, and $e^* \in P_T(u,v)$ such that
\begin{equation*}
    x_{e^*} = \max_{e \in P_T(u,v)} x_e \geq R_n^T(u,v).
\end{equation*}
Then we almost surely have $x_{e^*} > R_n^T(u,v)$. Hence, there is a path from $u$ to $v$ in $K_n \setminus E(T)$ whose edges all have weight strictly less than $x_{e^*}/n$. Together with the edges of $P_T(u,v) \setminus \{e^*\}$, which also have weight strictly less than $x_{e^*}/n$, this path connects the endpoints of $e^*$ before Kruskal's algorithm examines $e^*$. Therefore, $e^* \notin M_n$.

Conversely, suppose that the edge $e \in E(T)$ is rejected by Kruskal's algorithm, and let $T_{e}^-$ and $T_{e}^+$ be the two connected components of $T \setminus \{e\}$. Almost surely, by Lemma~\ref{L:MST_cycle}, there exists a cycle $C$ in $K_n$ containing $e$ where each edge other than $e$ has weight less than $w_{e}$. Furthermore, there must exist vertices $u \in V(T_{e}^-)$ and $v \in V(T_{e}^+)$ that are connected using $C$ without any of the edges in $T$. Consequently, $x_{e} \geq R_n^T(u,v)$ and therefore 
\begin{equation*}
    \P_n ( T \subseteq M_n \mid w_e = x_e/n \ \forall e \in E(T)) = \P_n(\max\limits_{e \in P_T(u,v)} x_e < R_n^T(u,v) \  \forall u\neq v ).
\end{equation*}

Let $t_1< t_2 < \ldots < t_k$ be the unique values in the set $\{ \max_{e \in P_T(u,v)} x_e : u,v \in V(T)\}$. In view of \eqref{eq:R_partition_equivalence}, the condition on the right-hand side of \eqref{eq:equiv_R^T} is completely determined by the partitions
\begin{equation*}
    (\Pi_n^{q,E(T)}(t_1), \ldots,  \Pi_n^{q,E(T)}(t_k) ).
\end{equation*}
By Lemma~\ref{L:partition_convergence}, these partitions converge in distribution to
\begin{equation*}
    (\Pi^{q}(t_1), \ldots,  \Pi^{q}(t_k) ).
\end{equation*}
In the limiting partition at time $t_i$, the vertices $u$ and $v$ lie in different blocks if and only if $\max(A_u,A_v) > t_i$. Thus,
\begin{align*}
    \limn \P_n(\max\limits_{e \in P_T(u,v)} x_e <R_n^T(u,v)  \ \forall u\neq v ) &= \P_A(\max_{e \in P_T(u,v)} x_e < \max(A_u, A_v) \ \forall u \neq v) \\
    &=\P_A(x_e < \tau_e(T) \ \forall e \in E(T)),
\end{align*}
where the second equality is Lemma~\ref{L:equiv_accept}.
\end{proof}

To pass from this pointwise limit to the integral over the edge weights of $T$, we use the following uniform tail bound.

\begin{lemma}\label{L:connection-tail}
Fix $m<\infty$. There are constants $C_m,n_m>0$ such that, for $n \geq n_m$, all trees $T$ with $|E(T)| \leq m$, $u, v \in V(T)$, $u \neq v$, and $y \geq 0$
\begin{equation}\label{eq:minimax-tail}
        \P_n\bigl(R_n^T(u,v)> y\bigr)\le C_m e^{-y/16}.
\end{equation}
\end{lemma}

\begin{proof}
Choose $n_m$ so large that
\begin{equation*}
    N=n-|V(T)| \geq n -( m+1) \geq 2n/3 \quad \forall n \geq n_m.
\end{equation*}
By increasing the constant $C_m$, it suffices to consider sufficiently large $y$ with $y<n$ (since $R_n^T(u,v)\leq n$ whenever $V(T) \neq [n]$). Let $W=[n]\setminus V(T)$, then the graph induced by $W$ in $G_n(y/n)$ has law $G_N(y/n)$. We first bound the probability that $G_N(y/n)$ does not contain a large component. If every component of $G_N(y/n)$ has size less than $3N/4$, then some union of components $S$ satisfies
\begin{equation*}
    \frac{N}{4} \leq |S|\leq\frac{3N}{4}.
\end{equation*}
For this event to occur, all $|S|(N - |S|)$ edges between $S$ and $S^c$ must be closed. Consequently, a union bound over the possible choices of $S$ gives
\begin{align*}
\P_n\Big(|\mathcal{C}_1^N(y/n)|<\frac{3N}{4}\Big)
&\leq \sum_{N/4\leq s\leq3N/4} \binom Ns (1-y/n)^{s(N-s)} \\
&\leq (1-y/n)^{3N^2/16} \sum_{N/4\leq s\leq3N/4} \binom Ns \\
&\leq 2^N\exp\left(-\frac{3yN^2}{16 n}\right) \leq e^{-y/16},
\end{align*}
where we used the inequalities $1-x \leq e^{-x}$, $N/n\geq2/3$ and $s(N-s)\geq 3N^2/16$ for $N/4 \leq s \leq 3N/4$, and that $y$ is large enough.

Outside this event, there is a component $\mathcal{C} \subseteq W$ with
\begin{equation*}
    |\mathcal{C}|\geq\frac{3N}{4}\geq\frac{n}{2}.
\end{equation*}
If both $u$ and $v$ have an open edge to $\mathcal{C}$ at time $p = y/n$, then they are connected by a path in $G_n(y/n) \setminus E(T)$. Conditional on the component $\mathcal{C}$, the two sets of $|\mathcal{C}|$ incident edges are independent of the graph induced by $W$. Hence,
\begin{equation*}
\P_n\bigl(R_n^T(u,v)>y\bigr) \leq e^{-y/16}+2(1-y/n)^{n/2} \leq 3e^{-y/16}. \qedhere
\end{equation*} 
\end{proof}

We now prove \eqref{eq:forest_formula} when $F = T$ is a single tree with $m$ edges. Conditioning on the edge weights of $T$ and making the change of variables $x_e = nw_e$ gives
\begin{equation*}
    n^m \P_n( T \subseteq M_n) = \int_{(0,n)^m} \P_n( T \subseteq M_n \mid w_e = x_e/n  \ \forall e \in E(T) ) \prod_{e \in E(T)} dx_e.
\end{equation*}
We may view this as an integral over $(0,\infty)^m$ by setting the conditional probability equal to zero when one coordinate $x_e$ is larger than $n$. By the observation in \eqref{eq:equiv_R^T}, if Kruskal's algorithm accepts all edges in $T$, then it must be 
\begin{equation*}
    x_e < R^T_n(e^-, e^+) \qquad \forall e \in E(T).
\end{equation*}
Therefore, H\"older's inequality with all exponents equal to $m$, together with Lemma~\ref{L:connection-tail}, gives
\begin{align*}
    \P_n( T \subseteq M_n \mid w_e = x_e/n  \ \forall e \in E(T) ) &\leq \prod_{e \in E(T)} \P_n(  R^T_n(e^-, e^+) > x_e)^{1/m} \\
    &\leq C_m \exp\Big( -\frac{\sum_{e \in E(T)} x_e}{16m} \Big).
\end{align*}
The right-hand side is integrable over $(0,\infty)^m$. Therefore, the dominated convergence theorem and Lemma~\ref{L:pointwise_conv} give
\begin{align*}
   \limn n^m \P_n( T \subseteq M_n) &= \int_{(0,\infty)^m} \limn \P_n( T \subseteq M_n \mid w_e = x_e/n  \ \forall e \in E(T) ) \prod_{e \in E(T)} dx_e \\
   &=\int_{(0,\infty)^m} \P_A(x_e < \tau_e(T) \ \forall e \in E(T))\prod_{e \in E(T)} dx_e.
\end{align*}
Hence, by Tonelli's theorem
\begin{align*}
    \int_{(0,\infty)^m} \P_A(x_e < \tau_e(T) \ \forall e \in E(T))\prod_{e \in E(T)} dx_e &= \E_A[ \int_{(0,\infty)^m} \prod_{e \in E(T)} \mathbf{1}_{x_e < \tau_e(T)} dx_e ] \\
    &= \E_A[ \prod_{e \in E(T)} \tau_e(T)].
\end{align*}

\subsection{Proof of Theorem~\ref{T:forest_formula} for forests}
We first prove the convergence of normalized tree counts in \eqref{eq:tree_count}. Their concentration will then allow us to pass from a single tree to a forest.

\begin{lemma} \label{L:tree_count}
    Let $T$ be a fixed tree.
    \begin{enumerate}[i)]
        \item \label{enu:E_tree_count} The normalized expected number of copies of $T$ satisfies
        \begin{equation*}
            \E_n\Big[\frac{N_n(T)}{n}\Big] \nconv \Psi(T).
        \end{equation*}
        \item \label{enu:Var_tree_count} Its variance satisfies
        \begin{equation*}
            \Var\left(\frac{N_n(T)}{n}\right) \nconv 0.
        \end{equation*}
        \item \label{enu:uniform_tree_count} For every $k \in \N$, its $k$-th moment is uniformly bounded:
        \begin{equation*}
            \sup_n \E_n\Big[\left(\frac{N_n(T)}{n}\right)^k\Big] < \infty.
        \end{equation*}
    \end{enumerate}
\end{lemma}
\begin{proof}
Let $m = |E(T)|$. Since $T$ has $m + 1$ vertices, vertex exchangeability gives
\begin{equation*}
    \frac{\E_n[N_n(T)]}{n} = \frac{(n)_{m+1}}{n} \P_n(T \subseteq M_n) =\frac{(n)_{m+1}}{n^{m+1}} (n^m \P_n(T \subseteq M_n)) \xrightarrow{n \to \infty} \Psi(T).
\end{equation*}

For item~\ref{enu:Var_tree_count}), we apply the Efron--Stein inequality, see e.g.\ \cite[Section~3.1]{BGM13}. Fix an edge $e \in E(K_n)$. First sample all edge weights to obtain $M_n$, and then resample only the weight of $e$ to obtain $M_n^{(e)}$. If $e$ belongs to neither $M_n$ nor $M_n^{(e)}$, then $M_n = M_n^{(e)}$. Hence,
\begin{equation} \label{eq:resample_bound}
    \P_n(M_n \neq M_n^{(e)}) \leq 2 \P_n(e \in M_n) = \frac{4}{n}.
\end{equation}
On the other hand, if $M_n \neq M_n^{(e)}$, then they can differ by at most two edges: exactly one of the two trees contains $e$, say $M_n$, and the other tree $M_n^{(e)}$ replaces $e$ with the edge $e^*$ that has the lowest weight connecting the two components of $M_n \setminus \{e\}$. 

We next bound the influence of resampling $e$ on $N_n(T)$. Write $\Delta_n := \max_{v \in V(K_n)}\deg_{M_n}(v)$ for the maximum degree in $M_n$. There are at most $2m\Delta_n^{m - 1}$ embeddings of $T$ into $M_n$ that contain $e$. Indeed, choose one of the $m$ edges of $T$ to map to $e$, choose one of its two orientations, and then expose the remaining $m - 1$ vertices in an order in which each new vertex is adjacent to an earlier one. At each step there are at most $\Delta_n$ choices. Let $\Delta_n^{(e)}$ and $N_n^{(e)}(T)$ denote the analogous quantities for $M_n^{(e)}$. It follows that, on $\{M_n \neq M_n^{(e)}\} \cap \{\Delta_n,\Delta_n^{(e)} \leq D\}$,
\begin{align*}
    |N_n(T) - N_n^{(e)}(T)| &\leq | \{\# \text{ embeddings of $T$ in $M_n$ containing $e$}  \} | \\
    &\qquad + | \{\# \text{ embeddings of $T$ in $M_n^{(e)}$ containing $e^*$}  \} | \leq 4m D^{m-1},
\end{align*}
where, by symmetry, we assumed that $e \in M_n$ and the replacement edge $e^*$ belongs to $M_n^{(e)}$. The Efron--Stein inequality therefore gives
\begin{align}
    \Var(N_n(T)) &\leq \frac{1}{2}\sum_{e \in E(K_n)} \E_n\bigl[(N_n(T) - N_n^{(e)}(T))^2\bigr] \nonumber \\
    &\leq 8m^2\sum_{e \in E(K_n)}\Bigl(D^{2m - 2}\P_n(M_n \neq M_n^{(e)}) + n^{2m - 2}\P_n(\Delta_n > D \text{ or } \Delta_n^{(e)} > D)\Bigr) \nonumber \\
    &\leq 4n^2m^2D^{2m - 2}\P_n(M_n \neq M_n^{(e)}) + 8m^2n^{2m}\P_n(\Delta_n > D), \label{eq:efron_var}
\end{align}
where $e$ is any fixed edge. 

For every fixed integer $k \geq 1$ and $n$ large enough depending on $k$, we have
\begin{equation*}
    \E_n[ (\deg_{M_n}(v))_{k}] =\E_n[(|B_{M_n}(1,1)|)_k]= (n-1)_k \P_n(S_k \subseteq M_n) \xrightarrow{n \to \infty} \Psi(S_k),
\end{equation*}
from which it follows that 
\begin{equation} \label{eq:deg_sup_bound}
    \sup_n \E_n[\deg_{M_n}(v)^k] < \infty.
\end{equation}
Hence, for every $k \in \N$, there exists a constant $C_k > 0$ such that
\begin{equation} \label{eq:degree_bound}
    \P_n( \Delta_n > D) \leq \sum_{v=1}^n \P_n(\deg_{M_n}(v) > D) \leq C_k n D^{-k}.
\end{equation} 
Using \eqref{eq:resample_bound} and \eqref{eq:degree_bound}, for any $k \in \N$, \eqref{eq:efron_var} is further upper bounded by
\begin{equation} \label{eq:var_combined}
    16 n m^2 D^{2m-2}  + 8m^2 n^{2m+1} C_k D^{-k}.
\end{equation}
Choose $\delta > 0$ and $k \in \N$ such that $(2m-2) \delta < 1$ and $k\delta > 2m$. Then letting $D = n^{\delta}$ in \eqref{eq:var_combined} shows that the variance of $N_n(T)$ is $o(n^2)$, proving item~\ref{enu:Var_tree_count}).

For item~\ref{enu:uniform_tree_count}), we first claim that
\begin{equation}
    N_n(T) \leq (m+1)\sum_{v\in V(K_n)}\deg_{M_n}(v)^m.
    \label{eq:embedding_count_v}
\end{equation}
Indeed, assign to each embedding $f:T\to M_n$ a vertex
$v_0\in f(V(T))$ having maximal degree in $M_n$ among the vertices
in $f(V(T))$. Fix $v_0\in V(K_n)$. There are $m+1$ choices for the vertex
$u_0\in V(T)$ satisfying $f(u_0)=v_0$. Enumerate the remaining vertices of $T$ as $u_1,\ldots,u_m$ in such a way that the parent of each $u_i$ appears before $u_i$. When choosing the vertex for $f(u_i)$, there are at most
\begin{equation*}
    \max_{0 \leq j < i}\deg_{M_n}\bigl(f(u_j)\bigr)
    \leq \deg_{M_n}(v_0)
\end{equation*}
possible choices. Repeating this for the remaining vertices and summing over all $v_0 \in V(K_n)$ gives \eqref{eq:embedding_count_v}.

For any $k \in \N$, Jensen's inequality\footnote{The deterministic version for convex functions.} gives
\begin{equation*}
    \big( \sum_{v\in V(K_n)}\deg_{M_n}(v)^m \big)^k = n^k\big( \frac{\sum_{v\in V(K_n)} \deg_{M_n}(v)^m}{n} \big)^k \leq n^{k-1}\sum_{v \in V(K_n) } \deg_{M_n}(v)^{km}.
\end{equation*}
Together with \eqref{eq:deg_sup_bound} and \eqref{eq:embedding_count_v} this yields
\begin{equation*}
    \E_n[N_n(T)^k] \leq (m+1)^k n^k \E_n[\deg_{M_n}(1)^{km}] \leq C_{k,m} n^k,
\end{equation*}
for some constant $C_{k,m} > 0$ depending only on $m = |E(T)|$ and $k$.
\end{proof}

Items~\ref{enu:E_tree_count}) and~\ref{enu:Var_tree_count}) imply that $N_n(T)/n$ converges to $\Psi(T)$ in probability. For any fixed $p < \infty$, choose an integer $k > p$. Item~\ref{enu:uniform_tree_count}) then gives a uniform bound on the $k$-th moment of $N_n(T)/n$, and hence uniform integrability of its $p$-th power. This proves the convergence in $L^p$ asserted in \eqref{eq:tree_count}.

\bigskip

Finally, let $F$ be a forest consisting of disjoint trees $(T_i)_{1 \leq i \leq k}$. Suppose that $T_i$ has $m_i$ edges, and let $m = \sum_{i=1}^k m_i$. Thus $F$ has $m$ edges and $m + k$ vertices. Let $N_n(F)$ be the number of embeddings of $F$ into $M_n$ for which the images of its components are vertex-disjoint. By exchangeability,
\begin{equation} \label{eq:forest_expectation}
    \frac{1}{n^k}\E_n[N_n(F)] = \frac{(n)_{m+k}}{n^k}\P_n(F \subseteq M_n) = (1+o_{k,m}(1)) n^{m} \P_n(F \subseteq M_n).
\end{equation}

The product $\prod_{i=1}^k N_n(T_i)$ counts all tuples $(f_i)_{1\leq i\leq k}$ of embeddings $f_i:T_i\to M_n$, without requiring their images to be vertex-disjoint. We now bound the number of tuples in $\prod_{i=1}^k N_n(T_i)$ that are not in $N_n(F)$. Fix a tuple $f=(f_i)_{1\leq i\leq k}$ and define a partition $\pi_f$ of $[k]$ by declaring $i$ and $j$ to be in the same block if there exist indices
\begin{equation*}
    i=i_0,i_1,\ldots,i_r=j
\end{equation*}
such that
\begin{equation*}
    V(f_{i_s}(T_{i_s}))\cap V(f_{i_{s+1}}(T_{i_{s+1}})) \neq\emptyset \qquad\text{for every }0\leq s<r.
\end{equation*}
If there is at least one overlapping embedding, then $\pi_f$ partitions $[k]$ into blocks $I_1, \ldots, I_\ell$, where $1 \leq \ell \leq k-1$. Furthermore, for each $1 \leq j \leq \ell$ the union
\begin{equation*}
    T_{I_j} := \bigcup_{i \in I_j} f_i(T_i)
\end{equation*}
is a finite connected tree with at most $m$ edges. For any such tree $T_{I_j}$, there are only finitely many ways in which the trees $(T_i)_{i \in I_j}$ can form $T_{I_j}$. Let $\mathcal{T}(m)$ be the finite set of trees with at most $m$ edges. Summing over all partitions of $[k]$ into at most $k - 1$ blocks, we obtain a constant $C_{k,m}$ such that
\begin{equation*}
    \prod_{i=1}^k N_n(T_i) - N_n(F) \leq C_{k,m}\sum_{\ell =1}^{k-1} \sum_{|\pi| = \ell} \prod_{I \in \pi}\sum_{T \in \mathcal{T}(m)} N_n(T),
\end{equation*}
where $|\pi|$ denotes the number of blocks in the partition $\pi$. For every partition $\pi$ with $|\pi| = \ell \leq k-1 $ we have
\begin{equation*}
    \frac{1}{n^\ell}\prod_{I \in \pi}\sum_{T \in \mathcal{T}(m)} N_n(T) = \Big( \sum_{T \in \mathcal{T}(m)} \frac{N_n(T)}{n} \Big)^\ell \xrightarrow{L^p} \Big( \sum_{T \in \mathcal{T}(m)} \Psi(T) \Big)^\ell \quad \forall 0 < p < \infty,
\end{equation*}
and since there are only finitely many such partitions (and $\ell \leq k-1$), it follows that 
\begin{equation*}
    \frac{1}{n^k}\E_n[\prod_{i=1}^k N_n(T_i) - N_n(F)] \nconv 0.
\end{equation*}

Therefore,
\begin{equation*}
    \frac{1}{n^k}\E_n[ N_n(F)] = \frac{1}{n^k}\E_n[\prod_{i=1}^k N_n(T_i)] + o_{k,m}(1).
\end{equation*}
As each $N_n(T_i)/n$ converges in $L^p$ to the constant $\Psi(T_i)$, equation \eqref{eq:forest_formula} follows from \eqref{eq:forest_expectation}.

% %%%%%%%%%%%%%%%%%%%%%%%%%%%%%%%%%%%%%%%%%%%
% %%%%%%%%%%%%%%%%%%%%%%%%%%%%%%%%%%%%%%%%%%%
% %%%%%%%%%%%%%%%%%%%%%%%%%%%%%%%%%%%%%%%%%%%

\section{Computing the function \texorpdfstring{$\Psi$}{Psi}} \label{S:compute_Psi}

We now turn to the calculation of $\Psi(T)$ for a general tree $T$.  Some routine details in the subsequent integral calculations are omitted. We begin by fixing the notation used throughout this section. Recall that
\begin{equation*}
\Psi(T) := \E_A\Big[\prod_{e \in E(T)} \tau_e(T)\Big],
\end{equation*}
where $\tau_e(T)$ is defined in \eqref{eq:def_tau}. The inverse function of $\theta(t)$ is
\begin{equation*}
\phi(t) := \frac{-\log(1 - t)}{t}, \qquad 0 < t < 1,
\end{equation*}
where we set $\phi(0) := 1$. The function $\phi$ is positive and increasing, with
\begin{equation*}
\phi'(t) = \frac{\frac{t}{1 - t} + \log(1 - t)}{t^2}.
\end{equation*}
The random variables in \eqref{eq:Law_A} used to define $\tau_e(T)$ may be constructed by setting
\begin{equation*}
A_v = \phi(U_v),
\end{equation*}
where $(U_v)_{v \in V(T)}$ are i.i.d.\ uniform random variables on $[0,1]$. We use this more convenient representation henceforth and omit subscripts from probabilities and expectations. 

Let $T_o$ be a tree rooted at $o$, and orient every edge away from the root. Set $T_o(o) = T_o$. For $x \in V(T_o) \setminus \{o\}$, let $T_o(x)$ be the descendant subtree rooted at $x$, obtained by removing the edge between $x$ and its parent and taking the component containing $x$. If $y_1,\ldots,y_d$ are the neighbors of $o$, then $T_o$ consists of the root $o$ with the branches $T_o(y_1),\ldots,T_o(y_d)$ attached to it. Define the non-increasing function
\begin{equation}
F_{T_o}(z) := \E\Big[\mathbf{1}_{\{U_v > z\ \forall v \in V(T_o)\}}\prod_{x \in V(T_o)}\phi\big(\min_{u \in V(T_o(x))} U_u\big)\Big].
\end{equation}

\subsection{General formula}

The following formula expresses $\Psi(T)$ in terms of the rooted branches obtained by choosing the vertex at which the minimum uniform variable is attained. It is the starting point to calculate $\Psi$ recursively.

\begin{lemma} \label{L:Psi_recursion}
    \begin{equation*}
        \Psi(T) = \sum_{o \in V(T)} \int_0^1 \prod_{y \sim o} F_{T_o(y)}(z) dz.
    \end{equation*}
\end{lemma}
\begin{proof}
    Let $Z := \min_{v \in V(T)} U_v$, and let $O$ be the almost surely unique vertex attaining this minimum. Fix $o \in V(T)$ and orient the edges of $T$ away from $o$. If $e$ joins a non-root vertex $x$ to its parent, then, on the event $\{O = o\}$,
    \begin{equation*}
        \tau_e(T) = \phi\left(\min_{u \in V(T_o(x))} U_u\right).
    \end{equation*}
    Indeed, the component of $T \setminus \{e\}$ containing $o$ has minimum $U_o$, whereas every uniform variable in the descendant component $T_o(x)$ is larger than $U_o$.

    Write $q := |V(T)|$. The variables $Z$ and $O$ are independent, $O$ is uniform on $V(T)$, and $Z$ has density $q(1 - z)^{q - 1}$ on $(0,1)$. Conditional on $\{O = o, Z = z\}$, the uniform variables in the branches $(T_o(y))_{y \sim o}$ are independent, and within each branch they are i.i.d.\ uniform on $(z,1)$. Therefore,
    \begin{equation*}
        \E\Big[\prod_{e \in E(T)} \tau_e(T) \ \big| \ \,O = o, Z = z\Big] = \prod_{y \sim o}\frac{F_{T_o(y)}(z)}{(1 - z)^{|V(T_o(y))|}}.
    \end{equation*}
    Since $\sum_{y \sim o}|V(T_o(y))| = q - 1$, conditioning on $Z$ and $O$ gives
    \begin{align*}
        \Psi(T) &= \sum_{o \in V(T)} \frac{1}{q}\int_0^1 q(1 - z)^{q - 1}\prod_{y \sim o}\frac{F_{T_o(y)}(z)}{(1 - z)^{|V(T_o(y))|}}\,dz \\
        &= \sum_{o \in V(T)} \int_0^1 \prod_{y \sim o} F_{T_o(y)}(z)\,dz. \qedhere
    \end{align*}
\end{proof}

By splitting into the branches of the root's neighbors, the functions $F_{T_o}$ may also be computed recursively. In the following, we define $(1-z) \phi(z) = 0$ at $z=1$.
\begin{lemma} \label{L:F_recursion}
 Suppose that $o$ has neighbors $y_1, \ldots, y_d$. Then
\begin{equation*}
    F_{T_o}(z) = (1 - z)\phi(z)\prod_{i=1}^d F_{T_o(y_i)}(z) + \int_z^1 (1 - t)\phi'(t)\prod_{i=1}^d F_{T_o(y_i)}(t)\,dt.
\end{equation*}
\end{lemma}
\begin{proof}
    Write $M := \min_{v \in V(T_o)} U_v$. On the event $\{M > z\}$,
    \begin{equation} \label{eq:phi_split}
        \phi(M) = \phi(z) + \int_z^1 \mathbf{1}_{\{M > t\}}\phi'(t)\,dt.
    \end{equation}
    As the branches $T_o(y_i)$ are disjoint, for $a \in V(T_o(y_i)),\, b \in V(T_o(y_j))$, $i \neq j$, the random variables $\phi(\min_{v\in V(T_o(a))} U_v)$ and $\phi(\min_{v\in V(T_o(b))} U_v)$ are independent. Inserting \eqref{eq:phi_split} into the definition of $F_{T_o}(z)$ gives
    \begin{align*}
        F_{T_o}(z) &= \E \big[ \mathbf{1}_{U_v > z \ \forall v \in V(T_o)} \phi(M) \prod_{x \in V(T_o) \setminus \{o\}} \phi(\min_{v \in V(T_o(x))} U_v) \big] \\
        &= \P(U_o > z) \phi(z) \prod_{y_i \sim o}\E \big[ \mathbf{1}_{U_v > z \ \forall v \in V(T_o(y_i))} \prod_{x \in V(T_o(y_i))} \phi(\min_{v \in V(T_o(x))} U_v) \big] \\
        &\qquad + \int_z^1 \P(U_o > t) \phi'(t) \prod_{y_i \sim o}\E\big[ \mathbf{1}_{U_v > t \ \forall v \in V(T_o(y_i))} \prod_{x \in V(T_o(y_i))} \phi(\min_{v \in V(T_o(x))} U_v) \big] dt.\\
        &=(1-z) \phi(z) \prod_{i=1}^d F_{T_o(y_i)}(z) + \int_z^1 (1-t)\phi'(t) \prod_{i=1}^d F_{T_o(y_i)}(t) dt
    \end{align*}
    as required.
\end{proof}
\noindent Using $(1 - z)\phi(z) + z(1 - z)\phi'(z) = 1$ and interchanging the order of integration, we also obtain
\begin{align} 
    \int_0^1 F_{T_o}(z)\,dz &= \int_0^1\left((1 - z)\phi(z) + z(1 - z)\phi'(z)\right)\prod_{i=1}^d F_{T_o(y_i)}(z)\,dz \nonumber \\
    &= \int_0^1 \prod_{i=1}^d F_{T_o(y_i)}(z)\,dz. \label{eq:integral_recursion}
\end{align}
We note the following useful consequence. If the chosen root $o$ is a leaf, then \eqref{eq:integral_recursion} can be iterated, successively removing vertices until a branch point is reached. This observation will be very useful for paths.

\subsection{Recursive calculations}
Lemmas~\ref{L:Psi_recursion} and~\ref{L:F_recursion} (and equation \eqref{eq:integral_recursion}) give the following recursive procedure for calculating $\Psi(T)$ for a specific tree $T$:
\begin{enumerate}
    \item Compute $F_{T_o}(z)$ for the required smaller rooted trees.
    \item Use Lemma~\ref{L:F_recursion} to obtain $F_{T_o}(z)$ for the larger rooted trees.
    \item Integrate to obtain $\Psi(T)$.
\end{enumerate}
The base case is the rooted tree consisting only of its root, which we denote by $L$, suppressing the notational dependence on the root. It satisfies
\begin{equation} \label{eq:F_leaf}
    F_{L}(z) = \E[\mathbf{1}_{U_o > z} \phi(U_o)]=\int_z^1\phi(t) dt,
\end{equation}
and can be considered as a building block for the recursion.  For $S_1 = P_1$, which consists of a single edge,
\begin{equation*}
    \Psi(S_1) = 2\int_0^1 F_L(z)\,dz = 2\int_0^1\int_z^1 \phi(t)\,dt\,dz = 2\int_0^1 t\phi(t)\,dt = 2.
\end{equation*}
If we root $S_1$ at either endpoint, Lemma~\ref{L:F_recursion} and integration by parts give (see also Appendix~\ref{S:Appendix})
\begin{align*}
    F_{S_1}(z) &= \frac{1}{2}F_L(z)^2 + \int_z^1 (1 - t)\phi(t)^2\,dt
\end{align*}
and by \eqref{eq:integral_recursion}
\begin{equation} \label{eq:leaf_integral}
    \int_0^1 F_{S_1}(z)dz = \int_0^1 F_L(z)dz = 1.
\end{equation}
For $S_2$, the root attaining the minimum is either the center, which produces two copies of $L$, or one of the two leaves, which produces a copy of the rooted tree $S_1$. In Appendix~\ref{S:Appendix}, we calculate
\begin{equation*}
    \Psi(S_2) =  \int_0^1F_L(z)^2dz + 2\int_0^1 F_{S_1}(z) dz = 8 - 4 \zeta(3).
\end{equation*}
We remark that Corollary~\ref{C:var_deg} now follows since
\begin{align*}
    \Var(\deg_{M_n}(1)) &=  \E_n[\deg_{M_n}(1)^2] - \E_n[\deg_{M_n}(1)]^2 \\
    &= (n-1)_2 \P_n(S_2 \subseteq M_n) + (n-1)\P_n(S_1 \subseteq M_n) - ( (n-1) \P_n(S_1 \subseteq M_n))^2 \\
    &\nconv \Psi(S_2) + \Psi(S_1) - \Psi(S_1)^2 =  6 - 4 \zeta(3).
\end{align*}

The same recursion gives explicit values for other small trees, and further examples are collected in Appendix~\ref{S:Appendix}. We remark that calculations become tedious very quickly, and the author is not aware of any general closed formula for arbitrary tree shapes. 

\subsection{Stars}
We now apply the recursion to the stars $(S_k)_{k \geq 1}$ where $S_k$ has $k$ edges and is implicitly rooted at its center. If the center of $S_k$ attains the minimum, the contribution is the product of $k$ copies of $F_L$. If one of the $k$ leaves attains the minimum, the remaining branch is $S_{k - 1}$ rooted at its center. Hence,
\begin{align}
    \Psi(S_k) &= \int_0^1 F_L(z)^k dz + k \int_0^1F_{S_{k-1}}(z) dz  \nonumber \\
    &=\int_0^1 F_L(z)^k dz + k \int_0^1F_L(z)^{k-1} dz \label{eq:star_formula},
\end{align}
where the second equality follows from \eqref{eq:integral_recursion}, since $S_{k - 1}$ rooted at its center has $k - 1$ branches, each equal to $L$. We claim that
\begin{equation*}
    \int_0^1 F_L(z)^k\,dz = (1 + o(1))\frac{\zeta(2)^{k + 1}}{k} \qquad \text{as } k \to \infty.
\end{equation*}
Indeed, $F_L(0) = \zeta(2)$, $F_L'(0) = -1$, and $F_L(z) < \zeta(2)$ for every $z > 0$, so that a form of Laplace's method (see e.g.\ \cite[Proposition~3.18]{Cos09}, where in their notation $f(z)=1$, $g(z) = \log F_L(z)$) gives the displayed asymptotic. Applying the asymptotic with $k$ and $k - 1$ in \eqref{eq:star_formula} yields $\Psi(S_k) = (1 + o(1))\zeta(2)^k$ as required in \eqref{eq:star_asymp} of Proposition~\ref{P:star_path}.

\subsection{Paths}
Let $P_k$ be the path with $k$ edges, with vertices labeled by $[k + 1]$. Unless specified, we implicitly root $P_k$ at one of the endpoints. If the minimum uniform variable occurs at $2 \leq i \leq k$, removing the vertex $i$ splits the path into two branches with $i - 2$ and $k - i$ edges. If it occurs at either endpoint, there is one branch equal to $P_{k-1}$. Lemma~\ref{L:Psi_recursion} therefore gives
\begin{equation} \label{eq:path_formula}
    \Psi(P_k) = 2\int_{0}^1 F_{P_{k-1}}(z)dz +\sum_{i=2}^k \int_0^1  F_{P_{i-2}}(z)  F_{P_{k-i}}(z) dz.
\end{equation}
Repeated application of \eqref{eq:integral_recursion} shows
\begin{equation*}
    \int_{0}^1 F_{P_{n}}(z)dz =  \int_{0}^1 F_{P_{n-1}}(z)dz = \ldots =  \int_{0}^1 F_{P_{0}}(z)dz =  \int_{0}^1 F_{L}(z)dz = 1.
\end{equation*}
We may therefore view $F_{P_{n}}(z)$ as a density\footnote{General tree shapes with branch points, or the path not rooted at one of the endpoints, are not necessarily normalized to integrate to $1$.} on $(0,1)$. By Lemma~\ref{L:F_recursion} for $n\geq 1$
\begin{equation} \label{eq:path_recursion}
    F_{P_{n}}(z) = (1-z) \phi(z) F_{P_{n-1}}(z) + \int_z^1 (1-t) \phi'(t) F_{P_{n-1}}(t) dt.
\end{equation}
We estimate the convolution-type integrals in \eqref{eq:path_formula} through the convergence of a Markov chain. We expect there to be a more direct analytic approach; however, we believe there is merit in showing this probabilistic approach.

Consider the Markov chain $(X_n)_{n \geq 0}$ on $(0,1)$, where $X_0$ has density $F_{P_0}(z) = F_L(z)$ and for $n\geq 0$
\begin{equation*}
    X_{n+1} =
    \begin{cases}
        X_n & \text{with probability } (1-X_n)\phi(X_n), \\
        U'_n X_n  & \text{with probability } X_n(1-X_n)\phi'(X_n)
    \end{cases}
\end{equation*}
where $U'_n$ are i.i.d.\ uniforms on $[0,1]$. Note that $X_{n+1} \leq X_n$ and
\begin{equation*}
    (1-z)\phi(z) +   z(1-z)\phi'(z) = 1 \quad \forall \ 0 < z < 1.
\end{equation*}
Equivalently, we have for any Borel $A \subseteq (0,1)$ and $t \in (0,1)$
\begin{align*}
    \P(X_{n+1} \in A \mid X_n = t) &=   (1-t) \phi(t) \mathbf{1}_A(t) + t (1-t)\phi'(t) \P(t U_n' \in A) \\
    &= (1-t) \phi(t) \mathbf{1}_A(t) + (1-t)\phi'(t) \int_0^t \mathbf{1}_A(z) dz.
\end{align*}
We claim that by induction $X_n$ has density $F_{P_n}(z)$. Indeed, for any Borel $A \subseteq (0,1)$, Lemma~\ref{L:F_recursion} gives
\begin{align*}
    \P(X_{n+1} \in A) &= \int_0^1 \P(X_{n+1} \in A \mid X_n = t) F_{P_n}(t) dt \\
    &= \int_0^1 \Big( (1-t) \phi(t) \mathbf{1}_A(t) + (1-t)\phi'(t) \int_0^t \mathbf{1}_A(z) dz \Big) F_{P_n}(t) dt \\
    &= \int_A(1-t) \phi(t)F_{P_n}(t) dt + \int_0^1 \int_z^1 (1-t)\phi'(t)F_{P_n}(t) \mathbf{1}_A(z) dt dz \\
    &=\int_A \Big((1-z) \phi(z)F_{P_n}(z) + \int_z^1 (1-t)\phi'(t)F_{P_n}(t) \mathbf{1}_A(z) dt \Big) dz \\
    &= \int_A F_{P_{n+1}}(z) dz.
\end{align*}

The following scaling limit is the key input for the path asymptotics.
\begin{lemma} \label{L:markov_conv}
    As $n \to \infty$,
    \begin{equation*}
    nX_n \xrightarrow{(d)} \mathrm{Exp}(1/2),
    \end{equation*}
    where $\mathrm{Exp}(1/2)$ denotes the exponential distribution with rate $1/2$.
\end{lemma}

\begin{proof}
    Since $X_{n + 1} \leq X_n$, there exists an almost sure limit $X_\infty \in [0,1]$. Fix $\eta > 0$ and suppose $X_\infty \geq \eta$. Whenever $X_n \in [\eta,1]$, the conditional probability that $X_{n + 1} \leq X_n/2$ is
    \begin{equation*}
        \frac{1}{2}X_n(1 - X_n)\phi'(X_n),
    \end{equation*}
    which is bounded below by a positive constant depending only on $\eta$. Consequently, such a halving will occur infinitely often, contradicting that $X_n \geq X_\infty \geq \eta$. Hence, $X_n \to 0$ almost surely. 
    
    For any $t \in (0,1)$ and $0 < x < t$, we have
    \begin{align*}
        \P(X_{n+1} > x \mid X_n = t) &= (1-t)\phi(t) + (1-t)\phi'(t) (t-x) \\
        &=1- x(1-t)\phi'(t) 
    \end{align*}
    and therefore
    \begin{equation} \label{eq:markov_recursion}
        \P(X_{n+1} > x \mid X_n) = \mathbf{1}_{X_n >x} ( 1- x(1-X_n)\phi'(X_n)).
    \end{equation}
    
    Notice that
    \begin{equation*}
        (1-t)\phi'(t) = \frac{1}{t} + \frac{1-t}{t^2} \log(1-t) = \frac{1}{2} + O(t) \quad \text{as } t\downarrow 0.
    \end{equation*}
    In particular, for any fixed $\varepsilon >0$ we may choose $\delta > 0$ such that 
    \begin{equation*}
        \frac{1}{2} - \varepsilon \leq (1-t)\phi'(t) \leq \frac{1}{2} + \varepsilon, \qquad 0 < t \leq \delta. 
    \end{equation*}
    For fixed $K \in \N$, define $A_K := \{ X_K \leq \delta\}$. By the monotonicity of $X_n$, this implies that $X_n \leq \delta$ for all $n \geq K$. By \eqref{eq:markov_recursion} for $n \geq K+1$
    \begin{align*}
         \P(X_{n} > \frac{y}{n}, A_K) &= \E[\mathbf{1}_{A_K} \P(X_{n} > \frac{y}{n} \mid \sigma(X_0, \ldots, X_{n-1}))]  \\
         &= \E[\mathbf{1}_{A_K} \mathbf{1}_{X_{n-1} >y/n}(1-\frac{y}{n}(1-X_{n-1})\phi'(X_{n-1}))]\\
         &\leq \big(1-\frac{y}{2n}(1-2\varepsilon) \big) \P(X_{n-1} > y/n, A_K) \\
         &\leq \big(1-\frac{y}{2n}(1-2\varepsilon)\big)^{n-K} \P(X_K > y/n, A_K)
    \end{align*}
    and similarly
    \begin{equation*}
        \P(X_{n} > \frac{y}{n}, A_K) \geq \big(1-\frac{y}{2n}(1+ 2\varepsilon)\big)^{n-K} \P(X_K > y/n, A_K).
    \end{equation*}
    Using the trivial bound
    \begin{equation*}
        \P(X_{n} > \frac{y}{n}, A_K) \leq \P(X_{n} > \frac{y}{n}) \leq \P(X_{n} > \frac{y}{n}, A_K) + \P(A_K^c)
    \end{equation*}
    together with the fact that $X_K > 0$ almost surely, %(so that $\P(X_K > y/n, A_K) \to \P(A_K)$)
    gives for any fixed $K$
    \begin{equation*}
        e^{-\frac{y}{2}(1+2\varepsilon)} \P(A_K) \leq \liminf_{n\to \infty}\P(X_{n} > \frac{y}{n}) \leq \limsup_{n\to \infty} \P(X_{n} > \frac{y}{n})\leq  e^{-\frac{y}{2}(1-2\varepsilon)} \P(A_K) + \P(A_K^c).
    \end{equation*}
    As $X_n \to 0$ a.s., letting $K$ go to infinity we obtain $\P(A_K) \to 1$. Since $\varepsilon$ was arbitrary, it follows that
    \begin{equation*}
        \limn \P(nX_{n} > y) = e^{-\frac{y}{2}}. \qedhere
    \end{equation*}
\end{proof}

Deducing the asymptotics of $\Psi(P_k)$ now becomes an exercise in analysis. We shall make use of the following.
\begin{lemma}[Stolz--Ces\`aro~{\cite[Section~12.6]{Wil91}}] \label{L:SC-asymp}
    If $(a_n)_{n \geq 1}, (b_n)_{n\geq 0}$ are sequences in $\R$ with $b_n\uparrow \infty$ and $a_n \to a_{\infty} \in \R$, then
    \begin{equation*}
        \frac{1}{b_n} \sum_{k=1}^n (b_k - b_{k-1}) a_k \nconv a_\infty.
    \end{equation*}
\end{lemma}

First, observe that $F_{P_n}$ is non-increasing for every $n\geq 0$. Indeed, $F_L$ and $ (1-z)\phi(z)$ are non-increasing, so that by \eqref{eq:path_recursion} and induction each $F_{P_n}$ is non-increasing. Hence, the density
\begin{equation*}
    \frac{1}{n}F_{P_n}\left(\frac{x}{n}\right)\mathbf{1}_{\{ 0 \leq x\leq n\}}
\end{equation*}
of $nX_n$ is non-increasing. With monotonicity and Lemma~\ref{L:markov_conv} one may obtain that
\begin{equation*}
     \frac{1}{n}F_{P_n}\left(\frac{x}{n}\right)\mathbf{1}_{\{0 \leq x\leq n\}} \to \frac{1}{2} e^{-\frac{x}{2}} \qquad \forall x >0.
\end{equation*}
Scheff\'e's lemma (see e.g.\ \cite[Section~5.1]{Wil91}) then yields
\begin{equation*}
    \varepsilon_n:=
    \int_0^\infty
    \left|
        \frac1nF_{P_n}\left(\frac{x}{n}\right)\mathbf{1}_{\{x\leq n\}}
        -\frac12e^{-x/2}
    \right|dx
    \longrightarrow0.
\end{equation*}
Moreover, since $0\leq(1-t)\phi'(t)\leq1$, the recursion in \eqref{eq:path_recursion} gives
\begin{equation*}
    F_{P_n}(0)\leq F_{P_{n-1}}(0)+1.
\end{equation*}
Together with the monotonicity, this implies
\begin{equation*}
    \sup_{n\geq1}\sup_{x\geq0}
    \frac1nF_{P_n}\left(\frac{x}{n}\right)\mathbf{1}_{\{0\leq x\leq n\}}
    <\infty.
\end{equation*}

We claim that there exists a constant $C > 0$ such that, for $m,\ell \geq 1$,
\begin{equation} \label{eq:path_overlap}
    \left|\int_0^1 F_{P_m}(z)F_{P_\ell}(z)\,dz - \frac{m\ell}{2(m + \ell)}\right| \leq C\bigl(\ell\varepsilon_m + m\varepsilon_\ell\bigr).
\end{equation}
Indeed, changing variables $z = x/(m + \ell)$, using $|ab - cd| \leq |a - c||b| + |c||b - d|$, and applying the uniform bound on the densities gives
\begin{align*}
    &\left|\int_0^1 F_{P_m}(z)F_{P_\ell}(z)\,dz - \frac{m\ell}{2(m + \ell)}\right|\\
    &\quad = \frac{m\ell}{m + \ell}\Bigg|\int_0^\infty \frac{\mathbf{1}_{\{x \leq m + \ell\}}}{m}F_{P_m}\left(\frac{x}{m + \ell}\right) \frac{\mathbf{1}_{\{x \leq m + \ell\}}}{\ell}F_{P_\ell}\left(\frac{x}{m + \ell}\right)\,dx\\
    &\qquad\qquad - \int_0^\infty \frac{e^{-mx/(2(m + \ell))}}{2}\frac{e^{-\ell x/(2(m + \ell))}}{2}\,dx\Bigg|\\
    &\leq C\frac{m\ell}{m + \ell}\int_0^\infty\left|\frac{\mathbf{1}_{\{x \leq m + \ell\}}}{m}F_{P_m}\left(\frac{x}{m + \ell}\right) - \frac{e^{-mx/(2(m + \ell))}}{2}\right|\,dx\\
    &\quad + C\frac{m\ell}{m + \ell}\int_0^\infty\left|\frac{\mathbf{1}_{\{x \leq m + \ell\}}}{\ell}F_{P_\ell}\left(\frac{x}{m + \ell}\right) - \frac{e^{-\ell x/(2(m + \ell))}}{2}\right|\,dx,
\end{align*}
for some universal constant $C  > 0$. Performing another change of variables $u = mx/(m+\ell)$ and $v = \ell x/(m+\ell)$ in the respective integrals, one obtains the upper bound
\begin{equation*}
    C \frac{m\ell}{m+\ell}\Big( \frac{m+\ell}{m} \varepsilon_m +  \frac{m+\ell}{\ell} \varepsilon_\ell \Big) = C (\ell \varepsilon_m + m \varepsilon_\ell).
\end{equation*}
Therefore, bounding the terms in \eqref{eq:path_formula} corresponding to $i \in \{2,k\}$ by constants gives with \eqref{eq:path_overlap},
\begin{align*}
    \Psi(P_k) &= 2 + \sum_{i=3}^{k - 1} \frac{(i - 2)(k - i)}{2(k - 2)} + O\Big(1 + \sum_{i=3}^{k - 1} \big((k - i)\varepsilon_{i - 2} + (i - 2)\varepsilon_{k - i} \big) \Big) \\
    &= \frac{k^2}{12} + O\Big(k + k\sum_{i=1}^{k - 3}\varepsilon_i\Big) = \frac{k^2}{12} + o(k^2),
\end{align*}
where the last equality follows from $\sum_{i=1}^k \varepsilon_i = o(k)$ by Lemma~\ref{L:SC-asymp}. This proves \eqref{eq:path_asymp} in Proposition~\ref{P:star_path}.

%%%%%%%%%%%%%%%%%%%%%%%%%

\section{Proof of Theorem~\ref{T:ball_size}} \label{S:concentration}

We first deduce \eqref{eq:ball_size} from $\Psi(P_k) = (1 + o(1))k^2/12$. For fixed $r \in \N$, \eqref{eq:ball_as_prob} and Theorem~\ref{T:forest_formula} give
\begin{equation*}
    \limn \E_n\bigl[|B_{M_n}(1,r)|\bigr] = \sum_{k=1}^r \Psi(P_k).
\end{equation*}
Lemma~\ref{L:SC-asymp} with $b_k = k^3$ and $a_k = \Psi(P_k)/(k^3 - (k-1)^3)$ gives
\begin{equation*}
\lim_{r \to \infty}\frac{\sum_{k=1}^r \Psi(P_k)}{r^3} = \lim_{r \to \infty}\frac{\Psi(P_r)}{r^3 - (r - 1)^3} = \frac{1}{36}.
\end{equation*}

\smallskip 

For $q \geq 2$ and $r \in \N$, let $\mathcal T_{q,r}$ be the set of isomorphism classes of finite trees equipped with $q$ distinct distinguished vertices labeled by $[q]$, with all remaining vertices unlabeled, such that
\begin{equation*}
    T = \bigcup_{i=2}^q P_T(1,i), \qquad d_T(1,i) \leq r, \quad 2 \leq i \leq q.
\end{equation*}
Isomorphisms are required to preserve the distinguished labels. Thus every vertex of $T$ lies on a path from the root $1$ to one of the distinguished vertices. Whenever such a tree appears in an inclusion probability, we choose an arbitrary labeling of its remaining vertices; by exchangeability, the probability is independent of this choice. Moreover,
\begin{equation*}
    |E(T)| \leq \sum_{i=2}^q |E(P_T(1,i))| \leq (q - 1)r,
\end{equation*}
so $\mathcal{T}_{q,r}$ is finite. We give an explicit bound on its size in \eqref{eq:T_rq_bound} below.

Every ordered $q$-tuple of distinct vertices in $B_{M_n}(1,r)$ determines a unique tree in $\mathcal{T}_{q + 1,r}$ by taking the union of its paths to $1$, and every embedding of such a tree determines one ordered $q$-tuple. Therefore, for all sufficiently large $n$ and fixed $q$,
\begin{equation}
    \E_n\big[(|B_{M_n}(1,r)|)_q\big] = \sum_{T \in \mathcal{T}_{q + 1,r}}(n - 1)_{|E(T)|}\P_n(T \subseteq M_n). \label{eq:ball_factorial}
\end{equation}
Consequently,
\begin{equation*}
    \limn \E_n\bigl[(|B_{M_n}(1,r)|)_q\bigr] = \sum_{T \in \mathcal{T}_{q + 1,r}}\Psi(T).
\end{equation*}
Thus, all factorial moments, and hence all moments, are uniformly bounded in $n$. By \cite{NT24} (see \eqref{eq:local_conv}), the sizes $|B_{M_n}(1,r)|$ converge in distribution to $|B_\mathcal{M}(\emptyset, r)|$, so that Corollary~\ref{C:converge_E} follows. In fact, all fixed moments of $|B_{M_n}(1,r)|$ converge to those of $|B_\mathcal{M}(\emptyset, r)|$. The following is the main estimate required for the tail bounds in Theorem~\ref{T:ball_size}.
\begin{lemma} \label{L:bound_sum_of_trees}
    For all $r,q \in \N$,
    \begin{equation*}
        \sum_{T \in \mathcal{T}_{q + 1,r}} \Psi(T) \leq (q + 1)!\,8^q r^{3q}. %q!(Kr^3)^q
    \end{equation*}
\end{lemma}
Before proving the lemma, we show how it implies \eqref{eq:ball_exp_bound} and \eqref{eq:ball_concentration}. By increasing $C$, it is enough to consider $\lambda \geq 128$. Set $q = \lfloor \lambda/64 \rfloor$ with
\begin{equation*}
    (\lambda r^3)_q = \prod_{j=0}^{q-1} (\lambda r^3 - j) \geq \Big( \frac{\lambda r^3}{2} \Big)^q.
\end{equation*}
By Lemma~\ref{L:bound_sum_of_trees} and Markov's inequality
\begin{align*}
    \limsup_{n \to \infty} \P_n(|B_{M_n}(1,r)| \geq \lambda r^3) &= \limsup_{n \to \infty} \P_n((|B_{M_n}(1,r)|)_q \geq (\lambda r^3)_q) \\
    &\leq \frac{(q+1)! 16^q r^{3q}}{\lambda^q r^{3q}} \leq (q+1)\Big(\frac{16q}{\lambda} \Big)^q \\
    &\leq (q+1)4^{-q} \leq 2^{-q},
\end{align*}
from which \eqref{eq:ball_exp_bound} follows. Taking $q = 1$ shows that $\E_n[|B_{M_n}(1,r)|] \leq 17r^3$ for all sufficiently large $n$. Hence, for $\lambda > 17$,
\begin{equation*}
    \left\{\left||B_{M_n}(1,r)| - \E_n|B_{M_n}(1,r)|\right| \geq \lambda r^3\right\} \subseteq \left\{|B_{M_n}(1,r)| \geq \lambda r^3\right\},
\end{equation*}
so that \eqref{eq:ball_concentration} follows from \eqref{eq:ball_exp_bound}.

\subsection{Proof of Lemma~\ref{L:bound_sum_of_trees}}

For a rooted tree $T_o$, let
\begin{equation*}
    h(T_o) := \max_{v \in V(T_o)} d_{T_o}(o,v)
\end{equation*}
be its height, and let $\ell(T_o)$ be the number of vertices with no children when all edges are oriented away from $o$. If $T_o$ is not the singleton and $o$ is not a leaf, then $\ell(T_o)$ counts the number of leaves of $T_o$.

\begin{lemma} \label{L:Psi_bound_L_D}
    Let $T$ be a non-trivial tree with $N$ vertices, $K$ leaves, and diameter $D$. Then
    \begin{equation*}
        \Psi(T) \leq N\zeta(2)^K D^{K - 1}.
    \end{equation*}
\end{lemma}
\begin{proof}
    We first claim that every rooted tree $T_o$, with $h := h(T_o)$ and $\ell := \ell(T_o)$, satisfies
    \begin{equation} \label{eq:F_induction_bound}
        \int_0^1 F_{T_o}(z)\,dz \leq \zeta(2)^\ell (h+1)^{\ell - 1}, \qquad F_{T_o}(0) \leq \zeta(2)^\ell(h + 1)^\ell.
    \end{equation}
    We argue by induction on $|V(T_o)|$. If $|V(T_o)| = 1$, then $h = 0$, $\ell = 1$, and \eqref{eq:F_leaf} and \eqref{eq:leaf_integral} give
    \begin{equation*}
        \int_0^1 F_{T_o}(z)\,dz = 1, \qquad F_{T_o}(0) = \zeta(2),
    \end{equation*}
    so the claim holds.

    Suppose that $|V(T_o)| \geq 2$ and $o$ has neighbors $y_1, \ldots, y_d$ with corresponding branches $T_o(y_1), \ldots, T_o(y_d)$. Write $h_i := h(T_o(y_i))$ and $\ell_i := \ell(T_o(y_i))$. Then $h_i + 1 \leq h$ and $\sum_{i=1}^d \ell_i = \ell$. Since each $F_{T_o(y_i)}$ is non-increasing, \eqref{eq:integral_recursion} and the induction hypothesis give
    \begin{align*}
        \int_0^1 F_{T_o}(z)\,dz &\leq \left(\int_0^1 F_{T_o(y_1)}(z)\,dz\right)\prod_{i=2}^d F_{T_o(y_i)}(0) \leq \zeta(2)^\ell h^{\ell - 1}.
    \end{align*}
    Moreover, Lemma~\ref{L:F_recursion}, \eqref{eq:integral_recursion}, and $(1 - t)\phi'(t) \leq 1$ yield
    \begin{align*}
        F_{T_o}(0) &\leq \prod_{i=1}^d F_{T_o(y_i)}(0) + \int_0^1 \prod_{i=1}^d F_{T_o(y_i)}(t)\,dt = \prod_{i=1}^d F_{T_o(y_i)}(0) + \int_0^1 F_{T_o}(t)\,dt\\
        &\leq \zeta(2)^\ell h^\ell + \zeta(2)^\ell h^{\ell - 1} \leq \zeta(2)^\ell(h + 1)^\ell.
    \end{align*}

    Now fix $o \in V(T)$. Each branch $T_o(y)$ has height at most $D - 1$, and every vertex with no children in such a branch is a leaf of $T$. Hence
    \begin{equation*}
        \sum_{y \sim o} \ell(T_o(y)) \leq K.
    \end{equation*}
    Choose one neighbor $y_o$ of $o$. Using the monotonicity of the functions $F_{T_o(y)}$ and \eqref{eq:F_induction_bound}, we obtain
    \begin{equation*}
        \int_0^1 \prod_{y \sim o} F_{T_o(y)}(z)\,dz \leq \left(\int_0^1 F_{T_o(y_o)}(z)\,dz\right)\prod_{\substack{y \sim o\\y \neq y_o}} F_{T_o(y)}(0) \leq \zeta(2)^K D^{K - 1}.
    \end{equation*}
    Summing this bound over $o \in V(T)$ in Lemma~\ref{L:Psi_recursion} completes the proof.
\end{proof}

We need one final bound on the number of trees in $\mathcal{T}_{q,r}$.
\begin{lemma}
    For every $r \geq 1$ and $q \geq 2$, we have
    \begin{equation}
        |\mathcal{T}_{q,r}| \leq 2^{q - 2}(q - 1)!\,r^{2q - 3} \label{eq:T_rq_bound}
    \end{equation}
\end{lemma}
\begin{proof}
    Fix $T\in\mathcal{T}_{q,r}$ and define
    \begin{equation*}
        T_j:=\bigcup_{i=2}^j P_T(1,i),
        \qquad 2\leq j\leq q,
    \end{equation*}
    with $T_q=T$. The tree $T_2=P_T(1,2)$ is determined by its length, which belongs to $\{1,\ldots,r\}$. Hence there are at most $r$ possibilities for
    $T_2$.
    
    Suppose that $T_j$ has been constructed for some $2\leq j<q$. Since
    $P_T(1,j+1)$ and $T_j$ are connected subtrees of $T$ containing $1$, their intersection is a path starting at $1$. Let $v_{j+1} \in V(T_j)$ be its other endpoint. The length of the path $P_T(v_{j+1},j+1)$ belongs to $\{0,\ldots,r\}$, where length zero means that $j+1$ is an existing vertex of $T_j$. Moreover,
    \begin{equation*}
        |V(T_j)|
        \leq 1+\sum_{i=2}^j d_T(1,i)
        \leq 1+(j-1)r
        \leq jr.
    \end{equation*}
    Consequently, given $T_j$, there are at most
    \begin{equation*}
        |V(T_j)|(r+1)\leq jr\cdot 2r=2jr^2
    \end{equation*}
    possibilities for $T_{j+1}$. It follows that
    \begin{equation*}
        |\mathcal{T}_{q,r}| \leq r\prod_{j=2}^{q-1}2jr^2=2^{q-2}(q-1)!r^{2q-3}. \qedhere
    \end{equation*}
\end{proof}

\begin{proof}[Proof of Lemma~\ref{L:bound_sum_of_trees}]
    Any tree $T \in \mathcal{T}_{q + 1,r}$ has at most $qr + 1$ vertices, diameter at most $2r$, and at most $q + 1$ leaves. Lemma~\ref{L:Psi_bound_L_D} gives
    \begin{equation*}
        \Psi(T) \leq (qr + 1)\zeta(2)^{q + 1}(2r)^q \leq 2(q + 1)4^q r^{q + 1}.
    \end{equation*}
    Summing over $T \in \mathcal{T}_{q + 1,r}$ and using \eqref{eq:T_rq_bound},
    \begin{equation*}
        \sum_{T \in \mathcal{T}_{q + 1,r}} \Psi(T) \leq 2^{q - 1}q!\,r^{2q - 1}\,2(q + 1)4^q r^{q + 1} = (q + 1)!\,8^q r^{3q},
    \end{equation*}
    as required.
\end{proof}

\subsection*{Statement on the use of AI}

ChatGPT (GPT-5.6 Sol) was used to assist in generating ideas for some of the mathematical arguments and to improve the language and presentation of this manuscript. It was particularly helpful in developing the code referenced in Appendix~\ref{S:Appendix}. The author takes full responsibility for all mathematical statements and any errors.

\bibliographystyle{plain}
\bibliography{MST_F}

\appendix

\section{Integral calculations} \label{S:Appendix}

Define the Riemann-zeta function and the dilogarithm as
\begin{align*}
    \zeta(s) &:= \sum_{n=1}^\infty \frac{1}{n^s}, \quad s > 1,\\
    \Li(z) &:= \int_0^z \phi(t) dt = \sum_{n=1}^\infty \frac{z^n}{n^2} \quad \text{with }\Li'(z) = \phi(z).
\end{align*}
We note the relations
\begin{equation} \label{eq:zeta_relations}
    \zeta(2)^2 = \frac{5}{2}\zeta(4), \qquad \zeta(2)\zeta(4) =\frac{7}{4} \zeta(6), \qquad \zeta(2)^3= \frac{35}{8} \zeta(6).
\end{equation}
In this notation, the leaf function from \eqref{eq:F_leaf} satisfies
\begin{equation} \label{eq:leaf_as_li}
    F_L(z) = \int_z^1 \phi(t) dt = \zeta(2) - \Li(z).
\end{equation}

\subsection{Values of \texorpdfstring{$\Psi$}{Psi} for \texorpdfstring{$S_2$}{Star2} and \texorpdfstring{$S_3$}{Star}}
In view of \eqref{eq:star_formula} and \eqref{eq:leaf_as_li}, for the stars, it suffices to find the integral of powers of $\Li(z)$. For many of the calculations, we will use the integrals from \cite{Fre05}. In particular, in Table~3 of \cite{Fre05} one has
\begin{equation} \label{eq:Li_2}
    \int_0^1 \Li^2(z)\,dz = 6-2\zeta(2)-4\zeta(3)+\frac52\zeta(4).
\end{equation}
With $\zeta(2)^2 = 5\zeta(4)/2$, one obtains
\begin{equation*}
    \Psi(S_2) = \Psi(P_2) = \int_0^1 F_L(z)^2 dz + 2\int_0^1F_L(z) dz = 8 - 4 \zeta(3).
\end{equation*}

In the following, $\IP$ indicates an application of integration by parts. The label $u$ marks the term to be differentiated, while $v'$ marks the term to be integrated. We have
\begin{align*}
\int_0^1
    \underbrace{\Li^3(z)}_{u}
    \underbrace{1}_{v'}\,dz
&\stackrel{\IP}{=}
\zeta^3(2)
+3\int_0^1
    \underbrace{\Li^2(z)}_{u}
    \underbrace{\log(1-z)}_{v'}\,dz                                    \\
&\stackrel{\IP}{=}
\zeta^3(2)
+6\int_0^1\frac{1-z}{z}\Li(z)
 \bigl(\log(1-z)-\log^2(1-z)\bigr)\,dz                                 \\
&=
\zeta^3(2)
+6\int_0^1\frac{\Li(z)\log(1-z)}{z}\,dz -6\int_0^1\frac{\Li(z)\log^2(1-z)}{z}\,dz  \\
&\qquad +6\int_0^1\Li(z)
 \bigl(\log^2(1-z)-\log(1-z)\bigr)\,dz .
\end{align*}
First,
\begin{equation*}
\int_0^1\frac{\Li(z)\log(1-z)}{z}\,dz
=
-\int_0^1 \Li(z) \phi(z)\,dz = -\frac{1}{2} \int_0^1 (\Li(z)^2)' dz=
-\frac12\zeta^2(2).
\end{equation*}
For the third integral above, the formula
\begin{equation*}
\int_0^1\frac{1-z}{z}\log^k(1-z)\,dz
= (-1)^k \sum_{n=1}^\infty \int_0^1 t^n (-\log t)^k dt = 
(-1)^k k!\bigl(\zeta(k+1)-1\bigr)
\end{equation*}
gives
\begin{align*}
&\int_0^1
 \underbrace{\Li(z)}_{u}
 \underbrace{\bigl(\log^2(1-z)-\log(1-z)\bigr)}_{v'}\,dz  \stackrel{\IP}{=}
-15+3\zeta(2)+6\zeta(3)+6\zeta(4).
\end{align*}
Finally, using the transformation $z=1-x$ with
\begin{equation} \label{eq:Li_transform}
    \Li(1-x)=\zeta(2)-\Li(x)-\log x\log(1-x)
\end{equation}
shows
\begin{align*}
\int_0^1\frac{\Li(z)\log^2(1-z)}{z}\,dz
&=
\zeta(2)\int_0^1\frac{\log^2x}{1-x}\,dx
-\int_0^1\frac{\log^2x\,\Li(x)}{1-x}\,dx \\
&\qquad
-\int_0^1\frac{\log^3x\log(1-x)}{1-x}\,dx  . 
\end{align*}
The first integral is $2 \zeta(3)$, and the other two integrals are given in Table~6 of \cite{Fre05}: 
\begin{align}
\int_0^1 \frac{\log^2x\,\Li(x)}{1-x}\,dx &= 6\zeta(2)\zeta(3)-11\zeta(5),  \label{eq:refer_table1}\\
\int_0^1 \frac{\log^3x\log(1-x)}{1-x}\,dx &= -6\zeta(2)\zeta(3)+12\zeta(5) \label{eq:refer_table2}.
\end{align}
Simplifying, we obtain 
\begin{equation} \label{eq:Li_3}
\int_0^1\Li^3(z)\,dz = -90+18\zeta(2)+36\zeta(3)+\frac{57}{2}\zeta(4) -12\zeta(2)\zeta(3)+6\zeta(5)+\frac{35}{8}\zeta(6).
\end{equation}
With  \eqref{eq:star_formula}, \eqref{eq:zeta_relations}, \eqref{eq:leaf_as_li}, \eqref{eq:Li_2} and \eqref{eq:Li_3} one can obtain
\begin{align*}
    \Psi(S_3) &= \int_0^1 F_L(z)^3 dz + 3\int_0^1F_L(z)^2 dz = 108 -48 \zeta(3) - 36\zeta(4) - 6\zeta(5).
\end{align*}

\subsection{The path \texorpdfstring{$P_3$}{Path3}}
Formulas \eqref{eq:path_formula} and \eqref{eq:path_recursion} give
\begin{align*}
    \Psi(P_3) = 2 + 2\int_0^1 F_L(z) F_{P_1}(z) dz,
\end{align*}
where, by Lemma~\ref{L:F_recursion},
\begin{align*}
    F_{P_1}(z) = (1-z) \phi(z)F_L(z) + \int_z^1 \underbrace{(1-t) \phi'(t)}_{v'} \underbrace{F_L(t)}_u dt \stackrel{\IP}{=} \frac{1}{2}F_L(z)^2 + \int_{z}^1(1-t) \phi(t)^2 dt.
\end{align*}
Therefore,
\begin{align*}
\Psi(P_3)
&= 2 + \int_0^1 F_L(z)^3\,dz + 2\int_0^1 F_L(z)\int_z^1 (1 - t)\phi(t)^2\,dt\,dz,
\end{align*}
where we calculated the integral of $F_L(z)^3$ previously. By Fubini,
\begin{equation*}
\int_0^1F_L(z)\int_z^1(1-t)\phi^2(t)\,dt\,dz = \int_0^1(1-t)\phi^2(t)
\left(\int_0^tF_L(z)\,dz\right)dt,
\end{equation*}
with
\begin{align*}
\int_0^t
\underbrace{F_L(z)}_{u}
\underbrace{1}_{v'}\,dz
\stackrel{\IP}{=}
tF_L(t)-\int_0^t zF_L'(z)\,dz=
tF_L(t)+t+(1-t)\log(1-t).
\end{align*}
Since $\phi(t)=-\log(1-t)/t$, it follows that
\begin{align}
\int_0^1F_L(z)\int_z^1(1-t)\phi^2(t)\,dt\,dz
&=
\int_0^1\frac{1-t}{t}\log^2(1-t)(\zeta(2) - \Li(t))\,dt \nonumber \\
&\quad+
\int_0^1\frac{1-t}{t}\log^2(1-t)\,dt
+
\int_0^1\frac{(1-t)^2}{t^2}\log^3(1-t)\,dt. \label{eq:Appendix_int1}
\end{align}
Using the substitution $x = 1-t$ and \eqref{eq:Li_transform}, the first integral becomes
\begin{align}
\int_0^1 \frac{1-t}{t}\log^2(1-t)\bigl(\zeta(2)-\Li(t)\bigr)\,dt
&= \int_0^1 \frac{\log^2x\,\Li(x)}{1-x}\,dx-\int_0^1 \log^2x\,\Li(x)\,dx \label{eq:Appendix_int2}\\
&\quad+\int_0^1 \frac{\log^3x\log(1-x)}{1-x}\,dx-\int_0^1 \log^3x\log(1-x)\,dx.  \nonumber
\end{align}
The first and third integrals on the right-hand side of \eqref{eq:Appendix_int2} are given in \eqref{eq:refer_table1} and \eqref{eq:refer_table2}. Moreover, integration by parts and the power-series expansion of $\log(1-x)$ give
\begin{align*}
\int_0^1 \underbrace{\log^2x}_{v'}\,\underbrace{\Li(x)}_{u}\,dx &\stackrel{\IP}{=} -12+6\zeta(2)+2\zeta(3), \\
\int_0^1 \log^3x\log(1-x)\,dx &= 24-6\zeta(2)-6\zeta(3)-6\zeta(4).
\end{align*}
Consequently,
\begin{equation*}
\int_0^1 \frac{1-t}{t}\log^2(1-t)\bigl(\zeta(2)-\Li(t)\bigr)\,dt
= -12+4\zeta(3)+6\zeta(4)+\zeta(5).
\end{equation*}
Using standard power-series expansions, one obtains the remaining two integrals in \eqref{eq:Appendix_int1}:
\begin{align*}
\int_0^1 \frac{1-t}{t}\log^2(1-t)\,dt &= 2\zeta(3)-2, \\
\int_0^1 \frac{(1-t)^2}{t^2}\log^3(1-t)\,dt
&= -6\sum_{k=3}^\infty \frac{k-2}{k^4}
= -6-6\zeta(3)+12\zeta(4).
\end{align*}
Combining these identities yields
\begin{equation*}
\int_0^1F_L(z)\int_z^1(1-t)\phi^2(t)\,dt\,dz
= -20+18\zeta(4)+\zeta(5).
\end{equation*}
One can now obtain
\begin{equation*}
\Psi(P_3)=52-36\zeta(3)-4\zeta(5).
\end{equation*}

%%%%%%%%%
\smallskip
\subsection{Numerical estimates}
Write $S_k(\ell)$ for the star with $k$ edges rooted at one of the leaves (and not at the center) and $P_k(i)$ for the path rooted at vertex $i\in [k+1]$. For the broom $B_4$ as in Figure~\ref{sfig:B4}, we have
\begin{align*}
    \Psi(B_4)&= 2 \int_0^1 F_{P_3(2)}(z)dz + \int_0^1 F_{L}(z)^2 F_{S_1}(z)dz + \int_0^1 F_{L}(z) F_{P_2(2)}(z)dz + \int_0^1 F_{S_3(\ell)}(z) dz \\
    &= 2 \int_0^1 F_{L}(z) F_{S_1}(z)dz+\int_0^1 F_{L}(z)^2 F_{S_1}(z)dz + \int_0^1 F_{L}(z) F_{P_2(2)}(z)dz + \int_0^1 F_{P_2(2)}(z) dz \\
    &=2 \int_0^1 F_{L}(z) F_{S_1}(z)dz+\int_0^1 F_{L}(z)^2 F_{S_1}(z)dz + \int_0^1 F_{L}(z) F_{P_2(2)}(z)dz + \int_0^1 F_{L}(z)^2 dz
\end{align*}
where
\begin{equation*}
    F_{P_2(2)}(z) = (1-z)\phi(z) F_{L}(z)^2 + \int_z^1 (1-t)\phi'(t) F_{L}(t)^2 dt.
\end{equation*}
These integral computations become increasingly complex, and we are not aware of a simple method for obtaining exact expressions. The accompanying code, available at \url{https://github.com/L-makowiec/mst-tree-psi-computations}, numerically computes $\Psi(T)$ for every tree shown in Figure~\ref{fig:trees_up_to_five_edges} and includes a function for computing $\Psi(T)$ for an arbitrary tree $T$. We expect that exact expressions can also be derived.

\end{document}